\documentclass[11pt]{article}
\usepackage[utf8]{inputenc}

\title{Spiked matrix}
\author{Tengyuan Liang, Subhabrata Sen, Pragya Sur}

\RequirePackage{amsthm,amsmath,amsfonts,amssymb,mathrsfs}
\RequirePackage{graphicx}
\RequirePackage[numbers]{natbib} 

\usepackage{authblk} 
\usepackage[dvipsnames]{xcolor}
\definecolor{Plum}{rgb}{.4,0,.4} 
\definecolor{BrickRed}{rgb}{0.6,0,0} 
\usepackage[plainpages=false,pdfpagelabels,colorlinks=true,
linkcolor=RedOrange,urlcolor=RedOrange,citecolor=MidnightBlue]{hyperref}
\providecommand{\keywords}[1]{\small \textbf{\textit{Keywords---}} #1}
\usepackage{fullpage}
\numberwithin{equation}{section}
\theoremstyle{plain}
\newtheorem{theorem}{Theorem}[section]
\newtheorem{lemma}[theorem]{Lemma}

\theoremstyle{remark}

\newtheorem{definition}[theorem]{Definition}

\newtheorem{remark}{Remark}

\def\ddefloop#1{\ifx\ddefloop#1\else\ddef{#1}\expandafter\ddefloop\fi}
\def\ddef#1{\expandafter\def\csname c#1\endcsname{\ensuremath{\mathcal{#1}}}}
\ddefloop ABCDEFGHIJKLMNOPQRSTUVWXYZ\ddefloop
\def\ddef#1{\expandafter\def\csname s#1\endcsname{\ensuremath{\mathsf{#1}}}}
\ddefloop ABCDEFGHIJKLMNOPQRSTUVWXYZ\ddefloop
\def\ddef#1{\expandafter\def\csname b#1\endcsname{\ensuremath{\mathbf{#1}}}}
\ddefloop ABCDEFGHIJKLMNOPQRSTUVWXYZ\ddefloop
\def\ddef#1{\expandafter\def\csname b#1\endcsname{\ensuremath{\mathbf{#1}}}}
\ddefloop abcdefghijklmnopqrstuvwxyz\ddefloop

\def\E{\mathop{\mathbb{E}}}

\def\Reals{\mathbb{R}}

\def\dif#1{\mathrm{d}{#1}\,}

\usepackage{bm,bbm}
\def\bsigma{\bm{\sigma}}
\usepackage{framed}

\newcommand{\iu}{{i\mkern1mu}}
\newcommand*\dd{\mathop{}\!\mathrm{d}}

\usepackage{marginnote} 

\begin{document}

\title{High-dimensional Asymptotics of Langevin Dynamics in Spiked Matrix Models}

\author[1]{Tengyuan Liang\thanks{\tt \href{mailto:tengyuan.liang@chicagobooth.edu}{tengyuan.liang@chicagobooth.edu}}}
\author[2]{Subhabrata Sen\thanks{\tt \href{mailto:subhabratasen@fas.harvard.edu}{subhabratasen@fas.harvard.edu}}}
\author[2]{Pragya Sur\thanks{\tt \href{mailto:pragya@fas.harvard.edu}{pragya@fas.harvard.edu}}}
\affil[1]{University of Chicago}
\affil[2]{Harvard University}

\maketitle

\begin{abstract}We study Langevin dynamics for recovering the planted signal in the spiked matrix model. We provide a ``path-wise" characterization of the overlap between the output of the Langevin algorithm and the planted signal. This overlap is characterized in terms of a self-consistent system of integro-differential equations, usually referred to as the Crisanti-Horner-Sommers-Cugliandolo-Kurchan (CHSCK) equations in the spin glass literature. As a second contribution, we derive an explicit formula for the limiting overlap in terms of the 
signal-to-noise ratio and the injected noise in the diffusion. This uncovers a sharp phase transition---in one regime, the limiting overlap is strictly positive, while in the other, the injected noise overcomes the signal, and the limiting overlap is zero.

\end{abstract}



\bigskip

\section{Introduction}
\label{sec:Introduction} 
Gradient descent based methods and their noisy counterparts are routinely used in modern Machine Learning. 
For a host of learning problems, it has now been established that
gradient based methods converge to special estimators with attractive generalization properties (\cite{soudry2018implicit},  \cite{chizat2020implicit},  \cite{gunasekar2018implicit}, \cite{gunasekar2018characterizing}, \cite{ma2018implicit}, \cite{ma2018power},  \cite{ji2019implicit}, \cite{nacson2019convergence}, \cite{arora2019implicit}, \cite{jin2020implicit}, \cite{zhang2021understanding}, \cite{chatterji2021does}). Thus the limiting performance of gradient descent and its variants can often  be characterized via  careful analyses of these special limiting estimators (c.f.~ \cite{montanari2019generalization}, \cite{belkin2019does}, \cite{deng2019model}, \cite{liang2020multiple}, , \cite{liang2020precise} , \cite{chatterji2021finite} and the references cited therein).
However, an understanding of ``path-wise" properties of these algorithms still lies in its infancy. In this paper, we consider Langevin dynamics  as a proxy for stochastic gradient descent, and a simple recovery problem with a non-convex objective function---that of recovering a planted rank 1 matrix under additive Gaussian noise.

Formally, we observe a symmetric matrix $\bJ \in \Reals^{N\times N}$, given by 
\begin{align}
	\bJ =  V V^\top  + \bZ, \label{eq:model} 
\end{align}
where $V \in \Reals^N$. 
We assume $\bZ = \bG^\top \bD \bG$ where $\bD$ is diagonal and $\bG$ is Haar distributed.  
Denote $\bD = {\rm Diag}(\sigma^1,\ldots, \sigma^N)$ to be a sequence of deterministic diagonal matrices. As $N$ increases, assume that $\mu_{D_N} = \frac{1}{N} \sum_{i=1}^{N} \delta_{\sigma^i}$ converges weakly to a probability measure $\mu_{D}$ with compact support. Let $d_+, d_{-}$ denote the upper and lower edges respectively of $\mu_{D}$. We assume throughout that $\max_{1\leq i \leq N} \sigma^i$ and $\min_{1\leq i \leq N} \sigma^i$ converge to $d_+$ and $d_{-}$ respectively. We assume that the entries of $V=(V^i)$ are iid $\mathcal{N}(0,\frac{\lambda}{N})$ for some $\lambda>0$.
We seek to recover the planted truth $V$, given the observation $\mathbf{J}$. The natural estimator in this setting is derived from PCA---one computes the eigenvector $\hat{v}$ corresponding to the largest eigenvalue of $\mathbf{J}$, and uses $\hat{v} \hat{v}^{\top}$ as an estimator for the latent subspace $VV^{\top}$. The performance of this estimator can be precisely characterized using recent advances in Random Matrix Theory (\cite{baik2005phase,benaych2011eigenvalues}). In particular, assume that the empirical spectral measure $\mu_{D_N}$ converges weakly to a limiting measure $\mu_{D}$ supported on $[d_{-},d_+]$. Define $G_{\mu_D}: \mathbb{R}\backslash [d_{-},d_+] \to \mathbb{R}$, 
\begin{align}
    G_{\mu_D}(z) = \mathbb{E}_{X \sim \mu_D}\Big[\frac{1}{z-X} \Big] \nonumber 
\end{align}
as the Cauchy transform of $\mu_D$. Further, define 
\begin{align}
    G_{\mu_D}(d^+) = \lim_{z \downarrow d_+} G_{\mu_D}(z),\,\,\,\,\, G_{\mu_D}(d^-) = \lim_{z \uparrow d_{-}} G_{\mu_D}(z). \nonumber 
\end{align}
The BBP phase transition establishes that if $\lambda > \frac{1}{G_{\mu_D}(d^+)}$, there exists an ``outlier" eigenvalue at $G_{\mu_D}^{-1}(1/\lambda)$; otherwise, the largest sample eigenvalue sticks to the spectral edge $d_+$. In the first case, the eigenvector $\hat{v}$ corresponding to the outlying eigenvalue has a non-trivial overlap with the planted signal $V$, i.e. with high probability, $|V^{\top}\hat{v}| >0$. In the latter case, $V^{\top}\hat{v} \stackrel{P}{\to}0$. 

To study the Langevin dynamics in this setting, we introduce the following 
system of Stochastic Differential Equations
\begin{align}\label{eq:Ld}
	\dif{X^i_t} = \sum_{j=1}^N \bJ_{ij} X_t^j \dif{t} - f'\big( \frac{1}{N} \sum_{j=1}^N (X_t^j)^2 \big) X_t^{i} \dif{t} + \beta^{-1/2} \dif{W_t^i} \enspace.
\end{align}
\noindent
We assume that $f:[0,\infty) \to \mathbb{R}$ satisfies $f'$ to be non-negative and Lipschitz. Eventually, we will apply our results to the special case $f(x)=x^2/2$. Note that this SDE can be looked upon as a penalized version of the PCA problem on taking $\beta \rightarrow \infty$. 
The function $f$ acts as a ``confining" potential, so that we can work without a norm constraint on the diffusion. For any $N \geq 1$ and $T\geq 0$, the SDE \eqref{eq:Ld} has a strong solution, which we denote by $\bm{X}_t \equiv \{X_t^i: 1\leq i \leq N, t\in [0,T]\}$ (see \cite[Lemma 6.7]{arous2001AgingSpherical} for a proof).  

The strong solution to the Langevin diffusion \eqref{eq:Ld} defines a natural collection of estimators indexed by $t\in [0,\infty)$. We track the statistical performance of these estimators via the normalized ``overlap" $R^N(t) = \frac{1}{N} \sum_{i=1}^{N} \sqrt{N} V^i X_t^i$. We note that the entries of $V$ are typically of order $1/\sqrt{N}$; the multiplicative factor $\sqrt{N}$ normalizes these entries so that the resultant is of order 1. Armed with these notations, we can pose a natural question of interest:

\begin{center}
\emph{How does the overlap $R^N(t)$ evolve over time? } 
\end{center}

In this paper, we provide sharp asymptotics for the evolution of $R^N(t)$ under a high-dimensional asymptotic limit where we send $N\to \infty$ with $t \geq 0$ \emph{fixed}. In practice, one should interpret these limits as approximate characterizers of the overlap when the dimension $N$ is large, while the diffusion has been run for a ``short time". In particular, we emphasize that the time scales involved are significantly shorter than those involved for ``mixing" of these diffusion processes. We believe that these asymptotics are particularly relevant for Statistics and Machine Learning. In particular, it allows one to explicitly characterize the statistical effect of ``early stopping" in this problem. 

Our first result characterizes the limiting behavior of $R^N(t)$. We will see that the behavior of $R^N$ is intricately tied to that of the ``auto-correlation" function $K^N(t,s) = \frac{1}{N} \sum_{i=1}^{N} X_t^i X_s^i$. Note that for any $T\geq 0$, $R^N$ and $K^N$ are sequences of (random) continuous functions on $[0,T]$ and $[0,T]^2$ respectively. In subsequent discussions, we equip these metric spaces with the sup-norm topology. 

In addition, we need to specify the initial conditions for the Langevin diffusion $\{\bX_t: t \geq 0\}$. We work under the following two sets of initial conditions.

\noindent 
\textbf{Initial Conditions:}
\begin{itemize}
    \item[(i)] (I.I.D. initial conditions) Assume that $\{(\sqrt{N} V^i,X_0^i): 1\leq i \leq N\}$ are i.i.d., independent of $\bG$ and $\{\bW_t: t \geq 0\}$. Assume furthermore that $\mathbb{E}[(X_0^i)^2]=1$, $\mathbb{E}[V^i X_0^i] = \frac{\rho \sqrt{\lambda} }{\sqrt{N}}$ for some $\rho \in [0,1]$. 
    \item[(ii)]  (I.I.D. under rotated basis) Define $\bY_t = \bG \bX_t$, $U= \bG V$, $\bB_t = \bG \bW_t$. Assume that $\{(Y_0^i, \sqrt{N} U^i): 1\leq i \leq N\}$ are i.i.d., independent of $\{\bB_t: t \geq 0\}$. Assume in addition that, $Y_0^i \sim F$ i.i.d,  $\mathbb{E}[(Y_0^i)^2]=1$, $\mathbb{E}[U^i Y_0^i] = \frac{\rho \sqrt{\lambda} }{\sqrt{N}}$ for some $\rho \in [0,1]$. 
\end{itemize}

\begin{remark}
Note that the parameter $\rho$ governs the correlation between the initialization and the spike direction. For concreteness, we assume that $\rho \geq 0$. The results generalize directly to $\rho<0$ if we replace $V$ by $-V$. In addition, we require the second moments of $X_0$ and $Y_0$ to be finite. The precise value 1 is chosen merely for convenience.  
\end{remark}

\begin{remark}
The I.I.D. initial condition is arguably the most natural initialization in Statistics and Machine Learning. The I.I.D. under rotated basis condition, although a bit less natural from this perspective, is simpler to analyze with the same theoretical tools. We think of the second initialization throughout as a warm-up, with the first initialization being of principal interest.  
\end{remark}

\begin{theorem}
\label{thm:limits} 
Assume one of the Initial Conditions specified above. Fix  $T\geq 0$.  As $N \to \infty$, $R^N$ and $K^N$ converge almost surely to deterministic limits $R$ and $K$ respectively. Furthermore, these limits are the unique solutions to  the following system of integro-differential equations: 
\begin{align}\label{eq:sys1}
	 R(t) & = \exp\left\{ -\int_0^t f'(K(s)) \dif{s} \right\} \cdot \E\limits_{(\bu, \bsigma, \bY_0)\sim \pi^{\infty}} \left[ \exp(\bsigma t) \bY_0  \bu  \right] \nonumber \\
	 &\qquad + \int_0^t  \exp\left\{ - \int_{s}^{t} f'(K(r)) \dif{r} \right\}  R(s)  \E\limits_{(\bu, \bsigma, \bY_0)\sim \pi^{\infty}} \left[ \exp(\bsigma(t-s) )\bu^2  \right]  \dif{s}. 
\end{align}

\begin{align}\label{eq:sys2}
	 &K(t, s) = \beta^{-1} \int_{0}^{s\wedge t} \exp\left\{ - \int_r^t f'(K(w)) \dif{w} - \int_r^s f'(K(w)) \dif{w} \right\} \E\limits_{(\bu, \bsigma, \bY_0)\sim \pi^{\infty}} \left[  \exp(\bsigma(t+s-2r))\right] \dif{r} \nonumber \\
	&\qquad  + \exp\left\{ -\int_0^t f'(K(r)) \dif{r} -\int_0^s f'(K(r)) \dif{r}  \right\} \cdot \E\limits_{(\bu, \bsigma, \bY_0)\sim \pi^{\infty}} \left[ \exp(\bsigma (t+s)) \bY_0^2 \right] \nonumber \\
	& \qquad + \exp\left\{ -\int_0^t f'(K(r)) \dif{r} \right\}  \cdot \int_0^s \exp\left\{ -\int_r^s f'(K(w)) \dif{w} \right\} R(r) \E\limits_{(\bu, \bsigma, \bY_0)\sim \pi^{\infty}} \left[ \exp(\bsigma (t+s-r)) \bY_0 \bu \right] \dif{r} \nonumber \\
	& \qquad + \exp\left\{ -\int_0^s f'(K(r)) \dif{r} \right\}  \cdot \int_0^t \exp\left\{ -\int_r^t f'(K(w)) \dif{w} \right\} R(r) \E\limits_{(\bu, \bsigma, \bY_0)\sim \pi^{\infty}} \left[ \exp(\bsigma (t+s-r)) \bY_0 \bu \right] \dif{r} \nonumber \\
	&  + \int_0^t \dif{r_1} \int_0^s \dif{r_2} \left\{ \exp\left\{ - \int_{r_1}^t f'(K(w)) \dif{w} -\int_{r_2}^s f'(K(w)) \dif{w}   \right\}  R(r_1) R(r_2) \E\limits_{(\bu, \bsigma, \bY_0)\sim \pi^{\infty}} \left[ \exp(\bsigma (t+s-r_1-r_2)) \bu^2 \right] \right\}.
\end{align}
Here we use the abbreviated notation $K(t):= K(t,t)$.
It remains to specify the distribution $\pi^{\infty}$; this limit depends on the initial condition: 
\begin{itemize}
    \item[(i)] Under I.I.D. initial conditions, $\pi^{\infty} = \mathcal{N}(0,1) \otimes \mu_D \otimes \mathcal{N}(0,1)$. 
    \item[(ii)] Under I.I.D. under rotated basis initialization, $\pi^{\infty} = \mathcal{N}(0,1) \otimes \mu_{D} \otimes F$. 
\end{itemize}
\end{theorem}

Theorem \eqref{thm:limits} holds for any $t \in [0,T]$. This allows one to characterize the overlap at any time $t$, for sufficiently large $N$. In this light, Theorem \eqref{thm:limits} may be interpreted as a quantification of the effects of early stopping on Langevin dynamics.  
However, though the theorem  provides a precise characterization of $R^N(t)$ under a high-dimensional limit, it is a priori unclear whether this description yields an explicit understanding regarding the behavior of $R^{N}(t)$. This is primarily due to the  mathematically involved nature of the fixed point equations \eqref{eq:sys1}-\eqref{eq:sys2}. To gain a better understanding of the CHSCK equations, our next theorem illustrates that Theorem \eqref{thm:limits} can yield explicit results on the correlation between the output $\bX_t$ and the latent vector $\sqrt{N}V$  under the double limit $t \to \infty$, following $N \to \infty$.
To this end, first note that this correlation can be captured through the ratio $R^2(t)/
(\lambda K(t,t))$, since  
\begin{align}
     \frac{\Big( \sum_{i=1}^{N} \sqrt{N} V^i  X_t^i\Big)^2}{[\sum_{i=1}^{N} (X_t^i)^2] [ N \sum_{i=1}^{N} (V^i)^2]}=\frac{(R^{N}(t))^2}{[K^N(t,t)][\sum_{i=1}^{N} (V^i)^2] }.  \nonumber 
\end{align}
and $\|V\|_2^2 \to \lambda$ almost surely under our initial conditions.

To precisely characterize the limiting correlation, we specialize to the following setting: consider $f(x) = x^2/2$ and $\mu_D$ to be the scaled semi-circle distribution, supported on $[-\sigma_{\star}, \sigma_{\star}]$, for some $\sigma_{\star} >0$. This corresponds to a setting where the additive noise $\bZ$ in \eqref{eq:model} is a symmetric Gaussian matrix. Note that $f(x)=x^2/2$ satisfies the regularity conditions required in Theorem \ref{thm:limits}. Formally, the semi-circle law on $[-\sigma_{\star}, \sigma_{\star}]$ has a density 
\begin{align}
    \frac{\dif{\mu_{D}}}{\dif{x}} = \frac{2}{\pi \sigma_{\star}^2}\sqrt{ \sigma_\star^2 - x^2}, \,\,\,\, -\sigma_{\star} \leq x \leq \sigma_{\star}. \label{eq:semi-circle}  
\end{align}
Let $\mathcal{S}:\mathbb{R} \backslash [-\sigma_{\star}, \sigma_{\star}] \to \mathbb{R}$ denote the Stieljes transform of $\mu_{D}$, i.e., 
\begin{align}
    \mathcal{S}(z) =  \E_{\bsigma\sim \mu_D}\big[ \frac{1}{z - \bsigma} \big] = \frac{2}{z + \sqrt{z^2 - \sigma_\star^2}}. \label{eq:stieljes} 
\end{align}

\begin{theorem}
\label{thm:limit_value}
Assume $\lambda > \sigma_{\star}/2$, and  set $\tilde{\lambda} = \lambda + \frac{\sigma_{\star}^2}{4 \lambda}$. If $\beta<\frac{1}{\sigma_{\star}^2}$, the equation $z= \beta^{-1} \mathcal{S}(z/2)$ has two real roots. Set $s_{\beta}$ to be the largest real root of this equation if $\beta< \frac{1}{\sigma_{\star}^2}$, otherwise set $s_{\beta}= 2\sigma_{\star}$. 
\begin{itemize}
    \item[(i)] If $2 \tilde{\lambda} < s_{\beta}$ or $\rho=0$, $\lim_{t \to \infty} \frac{R(t)^2}{\lambda K(t,t)} =0$. 
    \item[(ii)] If $2 \tilde{\lambda} > s_{\beta}$ and $\rho>0$,
    \begin{align}
    \lim_{t\to \infty} \frac{R(t)^2}{\lambda K(t,t)} =   \left( 1-\frac{\beta^{-1}}{2 \lambda(\lambda + \sigma_{\star}^2/4\lambda)}\right)\left( 1-\frac{\sigma_{\star}^2}{4\lambda^2}\right) .
    \end{align} 
\end{itemize}
\end{theorem}

Several remarks are in order regarding Theorem \ref{thm:limit_value}. First, note that $R^2(t)/
(\lambda K(t,t))$ captures the limiting \emph{correlation} between the output $\bX_t$ and the latent vector $\sqrt{N}V$, since  
\begin{align}
     \frac{\Big( \sum_{i=1}^{N} \sqrt{N} V^i  X_t^i\Big)^2}{[\sum_{i=1}^{N} (X_t^i)^2] [ N \sum_{i=1}^{N} (V^i)^2]}=\frac{(R^{N}(t))^2}{[K^N(t,t)][\sum_{i=1}^{N} (V^i)^2] }.  \nonumber 
\end{align}
and $\|V\|_2^2 \to \lambda$ almost surely under our initial conditions. 
Further, it is information theoretically impossible to recover the planted vector $V$ below the BBP threshold (i.e., when $\lambda < \sigma_{\star}/2$). Thus we focus on the interesting regime $\lambda > \sigma_{\star}/2$. In this setting, Theorem \ref{thm:limit_value}  can be interpreted as follows: first, if $\rho=0$ i.e., if the initialization has $o(1)$ correlation, then the limiting correlation stays at zero. We note that our asymptotics is \emph{different} from non-asymptotic results (in $N$) which allow vanishing (in $N$) overlap of the initialization with the truth. This is specifically the notion considered in \cite{arous2021online}, \cite{arous2020algorithmic}. We defer a detailed comparison of our results with this recent line of work to Section~\ref{sec:related}.  Our result establishes that if the diffusive noise is large ($2 \tilde{\lambda} < s_{\beta}$), the initial correlation washes off, and the correlation converges to zero in the limit $t\to \infty$.  On the other hand, if the diffusive noise is small ($2 \tilde{\lambda} > s_{\beta}$), we obtain a non-trivial correlation in the limit. While our setup is quite simple, we can precisely quantify the tradeoff between the SNR (captured by $\lambda$) and the strength of the injected noise (captured by $\beta^{-1}$) in the limiting correlation---we consider this to be one of the main contributions of our work. 
In particular, the limiting correlation increases as a function of $\lambda$, as well as $\beta^{-1}$, as one would naturally expect.

\subsection{Related literature}
\label{sec:related} 

\begin{itemize}
    \item[(i)] Dynamical Mean Field Theory and Crisanti-Horner-Sommers-Cugliandolo-Kurchan (CHSCK) equations: Dynamical Mean-Field Theory originated in the theory of spin glasses in the 80's \citep{sompolinsky1982relaxational, sompolinsky1981dynamic}. In this approach, dynamics are characterized in terms of the ``correlation" and ``response" functions. In special cases, these functions satisfy a system of integro-differential equations \citep{cugliandolo1993analytical, crisanti1992sphericalp} (we refer to such systems as CHSCK equations henceforth for simplicity). In general settings, the effective dynamics is described in terms of a non-Markovian stochastic process with long-term memory. 
    
    Recently, this framework has been employed in the statistical physics literature to study high-dimensional inference problems such as Gaussian mixture models, max-margin classification, and tensor PCA \citep{mannelli2019passed, mannelli2020marvels, cammarota2020afraid, mannelli2020thresholds, mignacco2020dynamical, sarao2021analytical, agoritsas2018out}. In these papers, versions of CHSCK equations have been proposed, and analyzed numerically to track the performance of specific algorithms. Our work differs from these earlier inquiries in some crucial ways---first, the derivations of the CHSCK equations in these works are non-rigorous, and the subsequent analysis is also numerical. In sharp contrast, our results are fully rigorous, furthermore, we characterize the precise tradeoff between the SNR and the injected noise in the Langevin algorithm. However, it should be noted that our model is, in some sense simpler than the other models described above. The main technical difference is that the spiked matrix model does not require a ``response function"---this considerably simplifies the CHSCK system, and the subsequent analysis.

    \item[(ii)] Prior rigorous results: Dynamical Mean-field Theory for mean-field spin glasses was established on rigorous footing in the works of \cite{arous2001AgingSpherical},  \cite{arous1998langevin,arous1997symmetric}, \cite{grunwald1996sanov}, \cite{grunwald1998sherrington}, among others. The CHSCK equations were formally derived in \cite{arous2004CugliandoloKurchanEquations}. While some useful information could be extracted from these equations under special settings (e.g. Langevin dynamics for matrix models \cite{arous2001AgingSpherical} and at high-temperature \cite{dembo2007limiting}), a general analysis of these equations has been quite challenging.   
    
    Recently, there has been renewed interest in this area. \cite{dembo2020dynamics}  derived the CHSCK equations for spherical spin glasses starting from disorder dependent initial conditions. \cite{dembo2021universality} examined the universality of these equations to the law of the disorder variables. \cite{dembo2021diffusions} also introduced alternative techniques for establishing universality of such dynamical algorithms. 
    
    We note that in sharp contrast to our setting, the models analyzed in this line of work correspond to ``null" models, i.e. without any planted signal. As a result, our analysis is not directly comparable to the aforementioned papers.  Despite this difference, our derivation of the CHSCK equations and its subsequent analysis relies partially on techniques from  \cite{arous2001AgingSpherical}, which studies the Langevin dynamics in matrix models  in the absence of a spike. 
    That said, we emphasize that planted models are closer to models typically observed in Statistics and Machine Learning problems. Thus, despite the simplicity of \eqref{eq:model}, Theorem \eqref{thm:limit_value} provides useful insights regarding the interplay between the SNR and the noise magnitude in determining the exact value of the limiting correlation. To the best of our knowledge, our work is the first to characterize this tradeoff for a planted model. We hope that this precise analysis would spark further investigations into Langevin-type dynamics for more complex planted models, as seen in contemporary Statistics and Machine Learning problems.

    \item[(iii)] Recent flow-based analyses: There has been considerable interest in understanding gradient descent algorithms (i.e. $\beta=\infty$ dynamics) in Statistics and Machine Learning. \cite{bodin2021model}, \cite{bodin2021rank} derive novel systems of integro-differential equations for gradient descent dynamics for certain spiked models. Although the equations are similar in spirit, there are crucial differences between these results, and the ones proved in this paper. First, for the spiked matrix model, the result in \cite{bodin2021rank} applies only to the $\beta=\infty$ case, and recovers the BBP phase transition. In contrast, we discover additional phase transitions depending on the strength of the injected noise. Further, the proof techniques are also completely different---\cite{bodin2021model,bodin2021rank} use recent advances in the random matrix literature on Green functions to derive their results.  In a different line of work, \cite{ali2019continuous}, \cite{ali2020implicit} have also used random matrix asymptotics to study the role of ``early stopping" for gradient descent algorithms for linear regression. In this problem, the solution at any fixed time is available in closed form, which considerably simplifies the analysis.   
    
   Recently, \cite{celentano2021high} derived the CHSCK equations for gradient descent (i.e. $\beta=\infty$) for empirical risk minimization. The approaches in these two papers are completely different---\cite{celentano2021high} first use Approximate Message Passing (AMP) style algorithms to study first order methods, and subsequently take a continuous time limit. Our approach, on the other hand, is grounded in random matrix theory.

   \item[(iv)] We note that our results assume a ``warm-start", i.e. an $O(1)$-correlation between the initialization and the planted signal. This situation is different from a random initialization, where the initialization has $O(1/\sqrt{N})$ correlation with the planted signal. The approach pursued in this paper does not seem suited to analyze the evolution of the correlation in this regime. Some answers are provided by the recent theory of ``bounding flows" \cite{benarous2020BoundingFlows} and the subsequent applications of this machinery to planted models \cite{arous2021online}, \cite{arous2020algorithmic}. The main restriction of this approach is that it yields ``non-sharp" answers, in constrast to the approach outlined in  our paper.   
    
\end{itemize}

\noindent 
\textbf{Outline:} The rest of the paper is structured as follows: we prove Theorem \ref{thm:limit_value}
in Section \ref{sec:proof_limit_value}, while the proof of the CHSCK equations Theorem \ref{thm:limits} is deferred to Section \ref{sec:CK_eqns}. 

\noindent \textbf{Acknowledgements} Liang was supported by NSF CAREER Grant DMS-2042473. Sur was supported by NSF DMS-2113426. Sen was supported by a Harvard Dean's Competitive Fund Award. 


\section{Proof of Theorem \ref{thm:limit_value}}
\label{sec:proof_limit_value} 

Throughout this section, for notational convenience, we use the shorthand $K(t)=K(t,t)$. Let $E(t) : = \exp\{ \int_{0}^t f'(K(s)) \dif{s} \}$. For our subsequent analysis, it will be convenient to transform the functions $R(t)$, $K(t)$ into a new set of functions $g(t)$, $h(t)$, defined as follows: 
\begin{align}
		g(t) &:=  E(t) R(t) \nonumber \\
 		h(t) &:= E^2(t) K(t). \label{eq:h_def} 
\end{align}
Observe that $(R(t))^2/K(t) = (g(t))^2/h(t)$, and thus it suffices to track $g(t)$, $h(t)$. In turn, we observe that \eqref{eq:sys1},\eqref{eq:sys2} imply that $g(t)$, $h(t)$ are uniquely specified as the solutions to the following fixed-point system. \begin{align}
		g(t) & = \mathbb{E}[\bY_0 \bu] \cdot \mathbb{E}[e^{t \bsigma}] + \mathbb{E}[\bu^2] \cdot \int_{0}^{t} g(s) \mathbb{E}[e^{(t-s) \bsigma}] \dif{s},\label{eq:gtterm}\\
		h(t) &= \beta^{-1} \int_0^t  E^2(s) \mathbb{E}[e^{(2t-2s) \bsigma}] \dif{s} + \mathbb{E}[\bY_0^2] \cdot \mathbb{E}[e^{2t \bsigma}] \\
		& + 2 \mathbb{E}[\bY_0 \bu] \cdot \int_0^t g(s) \mathbb{E}[e^{(2t-s) \bsigma}] \dif{s} + \mathbb{E}[\bu^2] \cdot \int_0^t \int_0^t g(s_1) g(s_2) \mathbb{E}[e^{(2t-s_1-s_2) \bsigma}] \dif{s_1} \dif{s_2}.  \label{eqn:most-ugly-term}
	\end{align}
Our first lemma characterizes the behavior of $g(\cdot)$. 

\begin{lemma}
\label{lemma:g_inversion} 
With $\bm{\sigma}$ drawn from the semi-circle distribution with parameter $\sigma_{\star}$, the function $g(\cdot)$ defined in \eqref{eq:h_def} satisfies the following
	\begin{align}\label{eq:g_exp}
		g(t) = \sqrt{\lambda }\rho \left\{ \Big(1- \frac{\sigma_\star^2}{4\lambda^2} \Big)_{+} \exp\big\{ (\lambda + \frac{\sigma_\star^2}{4\lambda}) t \big\}  + \frac{1}{2\pi \lambda} \int_{-\sigma_\star}^{\sigma_\star} \frac{e^{xt}\sqrt{\sigma_\star^2 - x^2}}{(\lambda+\frac{\sigma_\star^2}{4\lambda}) - x} \dd{x} \right\} \;,
	\end{align}
	where $x_{+} = \max \{ x,0\}$.
\end{lemma}

\begin{proof}[Proof of Lemma \ref{lemma:g_inversion} ]

Note that the last term in the RHS of \eqref{eq:gtterm} contains a convolution, suggesting that the Laplace transform of $g(\cdot)$ should provide useful information. For ${\rm Re}(z) > \sigma_\star$, recall that the Laplace transform of $g$ is given by
	\begin{align}\label{eq:laplace}
		\cL_g(z) := \int_{0}^\infty e^{-z t} g(t) \dif{t},
	\end{align}
	and the Stieltjes transform of $\mu_D$ is given by
	\begin{align}
		\cS(z) = \mathbb{E}\big[ \frac{1}{z - \bsigma} \big].
	\end{align}
	Evaluating Laplace transforms, Eqn.~\eqref{eq:laplace} can be transformed as 
	\begin{align}
		\cL_g(z) =  \frac{\E[\bY_0 \bu] \cdot \cS(z) }{ 1 - \E[\bu^2] \cdot \cS(z)}.
	\end{align}
	By the Fourier-Mellin formula (inverse Laplace transform), we know that
	\begin{align}\label{eq:gtformula}
		g(t) =  \frac{1}{2\pi i} \lim_{T \rightarrow +\infty} \int_{\gamma-iT}^{\gamma+iT} e^{zt} \cL_g(z) dz.
	\end{align}

\noindent 
	When $\mu_D$ follows a semi-circle law, using \eqref{eq:stieljes},  we have
	\begin{align}
		\cL_g(z) =  \frac{2\sqrt{\lambda} \rho  }{ z + \sqrt{z^2 -\sigma^2_\star} - 2\lambda } \;.
	\end{align}
	
	\begin{figure}%
	  \centering
	  \includegraphics[width=0.5\linewidth]{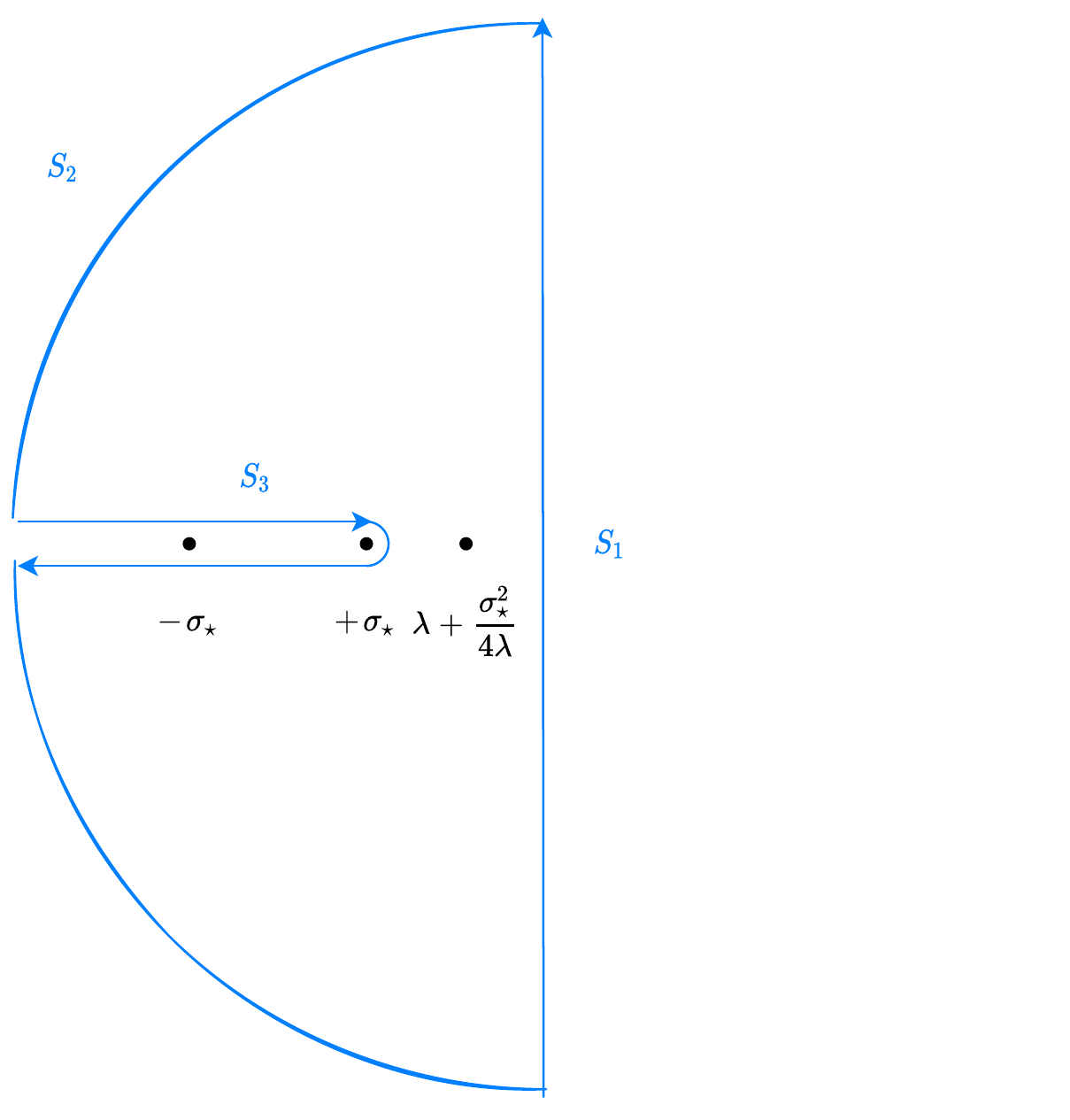}
	  \caption{The contour integral for evaluating the Fourier-Mellin formula.}
	  \label{fig:marginfig}
	\end{figure}

	Fix $\gamma \in \Reals$ with $\gamma > \lambda+ \frac{\sigma_\star^2}{4\lambda}$, we need to evaluate the Fourier-Mellin integral of the form
	\begin{align*}
		g_R(t)& :=  \frac{1}{2\pi \iu} \int_{\gamma-\iu R}^{\gamma+ \iu R} e^{zt} \cL_g (z) \dd{z} \\
		& =  \frac{1}{2\pi \iu} \int_{S_1}  e^{zt} \cL_g(z) \dd{z} \;.
	\end{align*}
	In other words, we need to know the integral along the line segment $S_1$ in Fig~\ref{fig:marginfig}.
	First, we observe that the function
	\begin{align}
		f_t(z)&:= e^{zt} \cL_g(z)  = e^{zt} \frac{2\sqrt{\lambda} \rho  }{ z + \sqrt{z^2 -\sigma^2_\star} - 2\lambda } \\
		&=  e^{zt} \frac{2\sqrt{\lambda}\rho}{4\lambda} \frac{\sqrt{z^2 -\sigma^2_\star} - (z-2\lambda) }{z - (\lambda+\frac{\sigma_\star^2}{4\lambda})}
	\end{align}
	has a simple pole at $z= \lambda+\frac{\sigma_\star^2}{4\lambda}$. Therefore using the Cauchy Residue Theorem, we have,
	\begin{align}
		\frac{1}{2\pi \iu}  \oint_{S_1 + S_2 + S_3} f_t(z) \dd{z} = \sqrt{\lambda} \rho  \cdot \Big(1-\frac{\sigma_\star^2}{4\lambda^2} \Big)_+  e^{(\lambda+\frac{\sigma_\star^2}{4\lambda})t}  \;.
	\end{align}
	
	Second, we evaluate the integral over the arc $S_2^{\epsilon}$. Parametrize the $z(\theta) \in S_2$ by $z(\theta) = \gamma - (R \cos\theta + \iu R \sin\theta)$ with $\theta \in [\frac{\pi}{2}+\epsilon, \frac{3\pi}{2}-\epsilon]$, therefore
	\begin{align}
		\left| \frac{1}{2\pi \iu} \int_{S_2^{\epsilon}} f_t(z) \dd{z} \right| \leq \frac{R}{2} e^{t(\gamma - R\cos(\epsilon))} \sup_{\theta \in [\frac{\pi}{2}+\epsilon, \frac{3\pi}{2}-\epsilon]} \frac{2\sqrt{\lambda}\rho}{| \gamma+R e^{\iu \theta} + \sqrt{(\gamma+R e^{\iu \theta})^2 - \sigma_\star^2} - 2\lambda |} 
	\end{align}
	and thus for any fixed $\epsilon$
	\begin{align}
		\lim_{R\rightarrow \infty} \frac{1}{2\pi \iu} \int_{S_2^{\epsilon}} f_t(z) \dd{z} = 0 \;.
	\end{align}
	For the part of the integral $S_2\backslash S_2^\epsilon$, one can show that
	\begin{align}
		\left| \frac{1}{2\pi \iu} \int_{S_2\backslash S_2^\epsilon} f_t(z) \dd{z} \right| \leq \epsilon \frac{R}{\pi} e^{t\gamma} \sup_{\theta \in [\frac{\pi}{2}, \frac{\pi}{2}+\epsilon]}  \frac{2\sqrt{\lambda}\rho}{| \gamma+R e^{\iu \theta} + \sqrt{(\gamma+R e^{\iu \theta})^2 - \sigma_\star^2} - 2\lambda |}
	\end{align}
	and thus with some universal constant $C(\lambda, \sigma_\star, \rho, \gamma)$
	\begin{align}
		\lim_{R\rightarrow \infty} \frac{1}{2\pi \iu} \int_{S_2^{\epsilon}} f_t(z) \dd{z} = C(\lambda, \sigma_\star, \rho, \gamma) \cdot \epsilon \;.
	\end{align}
	Putting these estimates together and sending $\epsilon \rightarrow 0$, we have shown 
	\begin{align}
		\lim_{R\rightarrow \infty} \frac{1}{2\pi \iu} \int_{S_2} f_t(z) \dd{z} = 0
	\end{align}

	Third, we evaluate the integral over the segment $S_3$ (corresponding to the branch point of $\sqrt{z^2 - \sigma_\star^2}$), with $z = x+ \iu \epsilon$ where $x \in [-\sigma_\star, \sigma_\star]$. We consider this since the function $z \mapsto \sqrt{z}$ is a well-defined single-valued function only on $z \in \mathbb{C} \backslash (-\infty, 0)$. One can show that
	\begin{align}
		&\lim_{\epsilon \rightarrow 0} \frac{1}{2\pi \iu} \int_{S_3^{\epsilon}} f_t(z) \dd{z} \\
		&= \frac{1}{2\pi \iu} \int_{-\sigma_\star}^{\sigma_\star} e^{x t} \frac{2\sqrt{\lambda}\rho}{(x-2\lambda)+ \iu\sqrt{\sigma_\star^2 - x^2} } \dd{x} - \frac{1}{2\pi \iu} \int_{-\sigma_\star}^{\sigma_\star} e^{x t} \frac{2\sqrt{\lambda}\rho}{(x-2\lambda)- \iu\sqrt{\sigma_\star^2 - x^2} } \dd{x} \\
		&= \frac{2\sqrt{\lambda}\rho}{2\pi \iu} \int_{-\sigma_\star}^{\sigma_\star} e^{xt}\frac{-2i \sqrt{\sigma_\star^2 - x^2}}{-4\lambda x + 4\lambda^2+\sigma_\star^2} \dd{x} \\
		&= - \sqrt{\lambda}\rho  \frac{1}{2\pi \lambda} \int_{-\sigma_\star}^{\sigma_\star} e^{xt}  \frac{\sqrt{\sigma_\star^2 - x^2}}{(\lambda+\frac{\sigma_\star^2}{4\lambda}) - x} \dd{x} \;.
	\end{align}

\noindent
Finally, we integrate over the semi-circular arc in $S_3$. Parametrizing $z= \sigma_{\star} + \epsilon e^{i\theta}$, we have, $\dif{z}= \epsilon i e^{i\theta} \dif{\theta}$. In turn, the integral reduces to 
\begin{align}
    \int_{\pi/2}^{-\pi/2} e^{(\sigma_{\star} + \epsilon e^{i \theta})t} \frac{2\sqrt{\lambda} \rho}{\sigma_{\star} + \epsilon e^{i \theta} + \sqrt{(\sigma_{\star} + \epsilon e^{i\theta} )^2 - \sigma_{\star}^2}- 2\lambda } \epsilon i e^{i \theta} \dif{\theta}. \nonumber  
\end{align}
We note that this integral converges to zero as $\epsilon \to 0$. Putting three pieces together, we conclude the proof.
\end{proof}

\noindent 
We next turn our attention to $h$. Recall $E(t) = \exp\{ \int_0^{t} K(s) \mathrm{d}s\}$, and set $F(t) = E^2(t)$. Differentiating, we obtain that $F'(t) = 2 K(t) F(t) = 2 h(t)$, defined in \eqref{eq:h_def}. Coupled with \eqref{eqn:most-ugly-term}, we have, $F$ satisfies the integro-differential equation, 
\begin{align}
    F'(t) &= 2 \beta^{-1} \int_0^{t} F(s) \mathbb{E}\Big[ \exp(2(t-s)\bm{\sigma}) \Big] \mathrm{d}s + \Phi(t), \nonumber \\
    \Phi(t) &= 2\mathbb{E}[\bm{Y_0}^2] \mathbb{E}[\exp(2t \bm{\sigma})] + 4\mathbb{E}[\bm{Y_0} \bm{u}] \int_0^t g(s) \mathbb{E}[\exp(\bm{\sigma}(2t-s))] \mathrm{d}s \label{eq:F_exp}  \\
    +& 2\mathbb{E}[\bm{u}^2] \int_0^t \int_0^t g(s_1) g(s_2) \mathbb{E}\Big[\exp\Big(\bm{\sigma} (2 t - s_1 - s_2 )\Big) \Big] \mathrm{d}s_1 \mathrm{d}s_2. \nonumber 
\end{align}

\begin{lemma}
\label{lem:h_laplace}
The function $h$ has Laplace transform \begin{align}
    \mathcal{L}_h(z) = \frac{1}{2} \Big[\frac{z \mathcal{L}_{\Phi}(z) +  \beta^{-1} \mathcal{S
    }(z/2) }{z -  \beta^{-1} \mathcal{S}(z/2)} \Big]. \nonumber
\end{align}
\end{lemma}

\begin{proof}[Proof of Lemma \ref{lem:h_laplace}]
Taking Laplace transforms in \eqref{eq:F_exp}, we obtain, 
\begin{align}
    z \mathcal{L}_{F}(z)-1 = 2 \beta^{-1} \mathcal{L}_F(z) \mathcal{L}_M(z) + \mathcal{L}_{\Phi}(z), \nonumber 
\end{align}
where we set $M(t)= \mathbb{E}[\exp(2t \sigma)]$. Transposing, we obtain, 
\begin{align}
    \mathcal{L}_{F}(z) = \frac{1+ \mathcal{L}_{\Phi}(z)}{z- 2 \beta^{-1} \mathcal{L}_{M}(z)}. \nonumber 
\end{align}
Recall that $F'(t) = 2 h(t)$, and thus $\mathcal{L}_{h}(z) = \frac{1}{2}[z\mathcal{L}_F(z) -1]$. This allows us to obtain the Laplace transform of $h$. Finally we note that by direct computation, $\mathcal{L}_{M}(z) = \frac{1}{2} \mathcal{S}(z/2)$. This completes the calculation. 

\end{proof}

We set $\tilde{\lambda}=\lambda + \frac{\sigma_{\star}^2}{4\lambda}$.  Let $s_{\beta}$ denote the largest real solution to the equation $z= \beta^{-1} \mathcal{S}(z/2)$.
\begin{lemma}
\label{lem:phi_inversion_asymp} 
If $2\tilde{\lambda}> s_\beta$, we have that
\begin{align}
    \lim_{z \to 0} \Big| z\mathcal{L}_{\Phi}(2 \tilde{\lambda} + z) - 2\lambda^2 \rho^2 \Big(1 - \frac{\sigma_{\star}^2}{4\lambda^2} \Big)_{+}^2\mathbb{E}\left[\frac{1}{(\tilde{\lambda}-\bm{\sigma})^2}\right] \Big| =0.\nonumber 
\end{align}
\end{lemma}

\begin{proof}
By definition, 
\begin{align}
    \mathcal{L}_{\Phi}(2 \tilde{\lambda} + z) &= \int_0^{\infty} \exp(- (2 \tilde{\lambda} + z) t) \Phi(t) \mathrm{d}t := T_1 + T_2 + T_3, \nonumber \\
    T_1 &:= 2\mathbb{E}[\bm{Y_0}^2] \int_0^{\infty} \exp(- (2 \tilde{\lambda} + z) t) \mathbb{E}[\exp(2t \bm{\sigma})] \mathrm{d}t, \nonumber \\
    T_2 &:= 4\mathbb{E}[\bm{Y_0} \bm{u}] \int_0^{\infty} \int_0^t \exp(- (2 \tilde{\lambda} + z) t) g(s) \mathbb{E}[\exp(\bm{\sigma}(2t-s))] \mathrm{d}s \mathrm{d}t, \nonumber \\
    T_3&:= 2\mathbb{E}[\bm{u}^2] \int_0^{\infty} \int_0^t \int_0^t \exp(- (2 \tilde{\lambda} + z) t) g(s_1) g(s_2) \mathbb{E}\Big[\exp\Big(\bm{\sigma}(2t-s_1 - s_2) \Big) \Big] \mathrm{d}s_1 \mathrm{d}s_2 \mathrm{d}t. \nonumber 
\end{align}
We first claim that there exists a constant $C>0$ such that 
\begin{align}
    \Big| \int_0^{\infty} \exp(- (2 \tilde{\lambda} + z) t) \mathbb{E}[\exp(2t \bm{\sigma})] \mathrm{d}t \Big|  \leq C, \,\,\,\, \Big|\int_0^{\infty} \int_0^t \exp(- (2 \tilde{\lambda} + z) t) g(s) \mathbb{E}[\exp(\bm{\sigma}(2t-s))] \mathrm{d}s \mathrm{d}t \Big|\leq C. \nonumber 
\end{align}
We establish each bound in turn. Set $z= z_1 + i z_2$ with $z_1, z_2 \in \Reals$, and note that $z_1, z_2 \to 0$. Choose $z_1$ small enough so that $2\tilde{\lambda} + z_1 > 4 \sigma_{\star}$ , since $2\tilde{\lambda} > 4 \sigma_\star$ by assumption. Thus we have,  
\begin{align}
    \Big| \int_0^{\infty} \exp(- (2 \tilde{\lambda} + z) t) \mathbb{E}[\exp(2t \bm{\sigma})] \mathrm{d}t \Big| &\leq \int_0^{\infty} \exp(- (2 \tilde{\lambda} + z_1) t) \mathbb{E}[\exp(2t \bm{\sigma})] \mathrm{d}t   \nonumber \\
    &\leq \int_0^{\infty} \exp\Big(- (2 \tilde{\lambda} - 2\sigma_{\star} + z_1) t \Big) \mathrm{d}t \nonumber \\
    &\leq \frac{1}{2\tilde{\lambda} -2\sigma_{\star} + z_1}. \nonumber 
\end{align}
This remains bounded as $z_1\to 0$. To analyze the second term, first observe from Lemma~\ref{lemma:g_inversion} that 
\begin{align}
    \lim_{t \to \infty} \exp(- \tilde{\lambda} t) g(t) = \sqrt{\lambda} \rho \Big( 1- \frac{\sigma_{\star}^2}{4\lambda^2}\Big)_{+}. \nonumber 
\end{align}

Thus for any $\varepsilon>0$, there exists $t_0$ large enough such that for all $t \geq t_0$, 
\begin{align}
    (1 - \varepsilon) \sqrt{\lambda} \rho \Big( 1- \frac{\sigma_{\star}^2}{4\lambda^2}\Big)_{+} \exp(\tilde{\lambda} t)\leq g(t) \leq (1 + \varepsilon) \sqrt{\lambda} \rho \Big( 1- \frac{\sigma_{\star}^2}{4\lambda^2}\Big)_{+} \exp(\tilde{\lambda} t). \nonumber
\end{align}
Then we have, setting $z = z_1 + i z_2$ with $z_1$ small enough, 
\begin{align}
   &\Big|\int_0^{\infty} \int_0^t \exp(- (2 \tilde{\lambda} + z) t) g(s) \mathbb{E}[\exp(\bm{\sigma}(2t-s))] \mathrm{d}s \mathrm{d}t \Big| \nonumber\\
   &\leq \int_0^{\infty} \int_0^t \exp(- (2 \tilde{\lambda} + z_1) t) g(s) \mathbb{E}[\exp(\bm{\sigma}(2t-s))] \mathrm{d}s \mathrm{d}t \nonumber \\
   &\leq C_1 + \int_{t_0}^{\infty} \Big[ \int_0^{t_0} \exp(- (2 \tilde{\lambda} + z_1) t) g(s) \mathbb{E}[\exp(\bm{\sigma}(2t-s))] \mathrm{d}s +  \int_{t_0}^{t} \exp(- (2 \tilde{\lambda} + z_1) t) g(s) \mathbb{E}[\exp(\bm{\sigma}(2t-s))] \mathrm{d}s \Big] \mathrm{d}t. \nonumber 
\end{align}
The control of this term is complete once we control the two integrals above. To this end, 
\begin{align}
    &\int_{t_0}^{\infty} \int_0^{t_0} \exp\Big(- (2 \tilde{\lambda} + z_1) t \Big) g(s) \mathbb{E}[\exp(\bm{\sigma}(2t-s))] \mathrm{d}s \mathrm{d}t \nonumber \\
    &\leq \Big(\max_{0\leq s \leq t_0} g(s) \Big) \int_{t_0}^{\infty} \exp\Big(- (2\tilde{\lambda} + z_1 -2\sigma_{\star}) t \Big) \mathrm{d}t \nonumber \\
    &\leq \Big(\max_{0\leq s \leq t_0} g(s) \Big) \frac{\exp(- t_0 (2 \tilde{\lambda} + z_1 - 2 \sigma_{\star}))}{2 \tilde{\lambda} + z_1 - 2 \sigma_{\star}} \nonumber 
\end{align}
which remains bounded as $z_1 \to 0$. Finally, 
\begin{align}
    &\int_{t_0}^{\infty} \int_{t_0}^{t} \exp\Big(- (2 \tilde{\lambda} + z_1) t) g(s) \mathbb{E}[\exp(\bm{\sigma}(2t-s))]  \Big) \mathrm{d}s\mathrm{d}t \nonumber \\
    &\leq (1 + \varepsilon) \sqrt{\lambda} \rho \Big( 1- \frac{\sigma_{\star}^2}{4\lambda^2} \Big)_{+} \int_{t_0}^{\infty} \int_{t_0}^{t} \exp\Big(- (2 \tilde{\lambda} + z_1) t) \Big) \exp(\tilde{\lambda}s+\sigma_{\star}(2t-s))   \mathrm{d}s\mathrm{d}t \nonumber \\
    &\leq \frac{(1+\varepsilon)}{\tilde{\lambda} -\sigma_{\star}} \sqrt{\lambda} \rho \Big( 1- \frac{\sigma_{\star}^2}{4\lambda^2} \Big)_{+} \int_{t_0}^{\infty} \exp\Big( - ( \tilde{\lambda} + z_1 - \sigma_{\star})t) \mathrm{d}t \nonumber 
\end{align}
which remains bounded as $z_1 \to 0$. 

Finally, we turn to the term $T_3$. Recall that, using \eqref{eq:g_exp}, $g(s)= \sqrt{\lambda} \rho (1- \frac{\sigma_{\star}^2}{4\lambda^2})_{+} \exp(\tilde{\lambda}s) + g_1(s)$, where 

\[g_1(s) = \frac{\sigma_{\star}^2}{4\lambda}\mathbb{E}_{x \sim \mu_D}\left[ \frac{e^{xs}}{\tilde{\lambda} - x}\right] .\]

Further recalling that $\mathbb{E}[\bm{u}^2]=\lambda$, we have
\begin{align}\label{eq:T3}
    T_3 =& 2\mathbb{E}[\bm{u}^2] \int_0^{\infty} \int_0^{t} \int_0^{t} \exp(- (2 \tilde{\lambda} + z) t) g(s_1) g(s_2) \mathbb{E}\Big[\exp\Big(\bm{\sigma}(2t-s_1 - s_2) \Big) \Big] \mathrm{d}s_1 \mathrm{d}s_2 \mathrm{d}t \nonumber\\
    &=  2\lambda^2 \rho^2 \Big(1 - \frac{\sigma_{\star}^2}{4\lambda^2} \Big)_{+}^2 \int_0^{\infty} \int_0^{t} \int_0^{t} \exp(- (2 \tilde{\lambda} + z) t + \tilde{\lambda}(s_1 + s_2))  \mathbb{E}\Big[\exp\Big(\bm{\sigma}(2t-s_1 - s_2) \Big) \Big] \mathrm{d}s_1 \mathrm{d}s_2 \mathrm{d}t + \mathrm{Rem} 
\end{align}
where $\mathrm{Rem}$ is the remainder term which we will later show to be bounded as $z\to 0$. Computing the first term, we obtain, 
\begin{align}\label{eq:rem1}
    &\int_0^{\infty} \int_0^{t} \int_0^{t} \exp(- (2 \tilde{\lambda} + z) t + \tilde{\lambda}(s_1 + s_2))  \mathbb{E}\Big[\exp\Big(\bm{\sigma}(2t-s_1 - s_2) \Big) \Big] \mathrm{d}s_1 \mathrm{d}s_2 \mathrm{d}t \nonumber \\
    &=\mathbb{E}\Big[\int_0^{\infty} \int_0^{t} \int_0^{t} \exp\Big(- (2 \tilde{\lambda} +z - 2 \bm{\sigma} )t + (\tilde{\lambda}- \bm{\sigma} ) s_1 + (\tilde{\lambda}- \bm{\sigma} ) s_2 \Big) \mathrm{d}s_1 \mathrm{d}s_2 \mathrm{d}t \Big] \nonumber \\
    &= \mathbb{E}\Big[ \int_0^{\infty} \exp(-(2\tilde{\lambda}+z - 2 \bm{\sigma})t) \frac{\Big( \exp(t(\tilde{\lambda} - \bm{\sigma})) - 1 \Big)^2}{(\tilde{\lambda}-\bm{\sigma})^2} \mathrm{d}t\Big] \nonumber \\
    &= \mathbb{E}\Big[ \int_0^{\infty} \exp(-(2\tilde{\lambda}+z - 2 \bm{\sigma})t) \frac{\Big( \exp(2 t(\tilde{\lambda} - \bm{\sigma})) - 2 \exp(t(\tilde{\lambda} - \bm{\sigma}))  +1    \Big)}{(\tilde{\lambda}-\bm{\sigma})^2} \mathrm{d}t \Big] \nonumber \\
    &= \frac{1}{z}\mathbb{E}\left[\frac{1}{(\tilde{\lambda}-\bm{\sigma})^2}\right]+ \mathrm{Rem}_1.
\end{align}

Note that $\mathrm{Rem}_1$ involves two terms and remains bounded when $z\to 0$ since

\begin{align*}
 \mathrm{Rem}_1 & = -2\mathbb{E} \left[\frac{1}{(\tilde{\lambda}-\bm{\sigma})^2} \int_{0}^{\infty}\exp((-\tilde{\lambda}-z+ \bm{\sigma})t)\right]  + \mathbb{E} \left[\frac{1}{(\tilde{\lambda}-\bm{\sigma})^2} \int_{0}^{\infty}\exp((-2\tilde{\lambda}-z+2\bm{\sigma})t)\right]   \\
 &= -2\mathbb{E} \left[\frac{1}{(\tilde{\lambda}-\bm{\sigma})^2(\tilde{\lambda}+z - \bm{\sigma})}  \right] + \mathbb{E} \left[\frac{1}{(\tilde{\lambda}-\bm{\sigma})^2(2\tilde{\lambda}+z -2 \bm{\sigma})}  \right]
\end{align*}
Thus, the proof is complete on establishing that $\mathrm{Rem}$ remains bounded as $z \rightarrow 0$.

We note that $\mathrm{Rem}$ may be expressed as the sum of three terms, denoted as $\mathrm{Rem}_{T_j}, j=1,\hdots 3$, where 

\begin{align*}
    \mathrm{Rem}_{T_1} & = 2\sqrt{\lambda}\rho(1-\frac{\sigma_{\star}^2}{4\lambda^2})_{+}\frac{\sigma_{\star}^2}{4\lambda} \int_0^{\infty} \int_0^{t} \int_0^{t} \exp(- (2 \tilde{\lambda} + z) t+\tilde{\lambda}s_1) \mathbb{E}_{x \sim \mu_D}\left[\frac{\exp(xs_2)}{\tilde{\lambda} -x}  \right] \mathbb{E}\Big[\exp\Big(\bm{\sigma}(2t-s_1 - s_2) \Big) \Big] \mathrm{d}s_1 \mathrm{d}s_2 \mathrm{d}t,
    \end{align*}
$\mathrm{Rem}_{T_2}$ is the same as above except the roles of $s_1$ and $s_2$ are reversed, and the third term is defined as below:

\begin{align*}
    \mathrm{Rem}_{T_3} =  \frac{2\sigma_{\star}^4}{16\lambda}\int_0^{\infty} \int_0^{t} \int_0^{t}\exp(- (2 \tilde{\lambda} + z) t\mathbb{E}_{x_1 \sim \mu_D}\left[\frac{\exp(x_1s_1)}{(\tilde{\lambda}-x_1)}\right] \mathbb{E}_{x_2 \sim \mu_D}\left[\frac{\exp(x_2s_2)}{(\tilde{\lambda}-x_2)}\right]\mathbb{E}\Big[\exp\Big(\bm{\sigma}(2t-s_1 - s_2) \Big) \Big] \mathrm{d}s_1 \mathrm{d}s_2 \mathrm{d}t
\end{align*}
  
  Now the support of $x , \bm{\sigma}$ is upper bounded by $\sigma_{\star}$ and $2t -s_1 -s_2 \geq 0, s_1,s_2 \leq t$ in the range of integration. Thus, we may upper bound $\mathrm{Rem}_{T_1}$ by
  
    \begin{align*}
    Rem_{T_1} & \leq  2\sqrt{\lambda}\rho(1-\frac{\sigma_{\star}^2}{4\lambda^2})_{+}\mathbb{E}_{(x,\bm{\sigma}) \sim \mu_D^{\otimes 2}}\left[\frac{1}{\tilde{\lambda}-x} \int_0^{\infty} \int_0^{t} \int_0^{t}\exp(- (2 \tilde{\lambda} + z-2 \sigma_{\star}) t+(\tilde{\lambda}-\sigma_{\star})t) ds_1 ds_2 dt\right] \\
    & = 2\sqrt{\lambda}\rho(1-\frac{\sigma_{\star}^2}{4\lambda^2})_{+}\mathbb{E}_{x \sim \mu_D}\left[\frac{1}{(\tilde{\lambda}-x)} \int_0^{\infty} t^2\exp(- ( \tilde{\lambda} + z-\sigma_{\star}) t)\right] \\ 
    & = 2\sqrt{\lambda}\rho(1-\frac{\sigma_{\star}^2}{4\lambda^2})_{+}\frac{\Gamma(3)}{(\tilde{\lambda}+z -\sigma_{\star})^3}\mathbb{E}_{x \sim \mu_D}\left[\frac{1}{(\tilde{\lambda}-x)}\right],
\end{align*}
which remains bounded when $z \rightarrow 0$. Similarly, $\mathrm{Rem}_{T_2}$ can be bounded. Applying a similar trick, $\mathrm{Rem}_{T_3}$ can be bounded by 

\begin{align*}
   & \frac{2\sigma_{\star}^4}{16\lambda}\mathbb{E}_{(x_1,x_2) \sim \mu_D^{\otimes 2}}\left[\frac{1}{(\tilde{\lambda}-x_1)(\tilde{\lambda}-x_2)}\right]\int_{0}^{\infty} t^2\exp(-(2\tilde{\lambda}+z-2\sigma_{\star})t)dt = \\ & \frac{2\sigma_{\star}^4}{16\lambda} \frac{\Gamma(3)}{(2\tilde{\lambda}+z -2\sigma_{\star})^3} \mathbb{E}_{(x_1,x_2) \sim \mu_D^{\otimes 2}}\left[\frac{1}{(\tilde{\lambda}-x_1)(\tilde{\lambda}-x_2)}\right],
\end{align*}
and this once again remains bounded as $z \rightarrow 0$. Putting everything together and from \eqref{eq:rem1}, we have the desired result.

\end{proof}

We now turn to the proof of Theorem \ref{thm:limit_value}. 

\begin{proof}[Proof of Theorem \ref{thm:limit_value}]

First, using \eqref{eq:stieljes}, we have, 
\begin{align}
    \lim_{z \downarrow \sigma_{\star} } \mathcal{S}(z) = \frac{2}{\sigma_{\star}}, \,\,\,\,\, \lim_{z \uparrow -\sigma_{\star} } \mathcal{S}(z) = -\frac{2}{\sigma_{\star}}. \nonumber 
\end{align}
If $z_{*}$ is a real root of $z= \beta^{-1} \mathcal{S}(z/2)$, $y=2z_*$ is a real root of the fixed point equation $2 \beta y = \mathcal{S}(y)$. In turn, such a root exists if and only if $2\beta \sigma_* < 2/ \sigma_*$, which immediately gives us the desired conclusion.

Consider first the regime $2 \tilde{\lambda} > s_{\beta}$. 
Observe that \eqref{eq:g_exp} implies,  
\begin{align}
    \lim_{t\to \infty} \exp(-2\tilde{\lambda} t) g^2(t) = \lambda \rho^2 \Big(1- \frac{\sigma_{\star}^2}{4\lambda^2} \Big)_{+}^2. \nonumber 
\end{align}
Further, Lemma \ref{lem:phi_inversion_asymp} implies that as $z \to 0$ along a sector, 
\begin{align}
  z \mathcal{L}_{\Phi}(2 \tilde{\lambda} + z) \to 2\lambda^2 \rho^2 \Big( 1 - \frac{\sigma_{\star}^2}{4\lambda^2} \Big)_{+}^2\mathbb{E}\left[\frac{1}{(\tilde{\lambda}-\bm{\sigma})^2}\right] \nonumber 
\end{align}
Thus Lemma \ref{lem:h_laplace} implies 
\begin{align}
  \lim_{z \to 0}  z\mathcal{L}_h(2 \tilde{\lambda} + z) = \frac{2\tilde{\lambda}}{2 \tilde{\lambda} - \beta^{-1} \mathcal{S} (\tilde{\lambda}) }\mathbb{E}\left[\frac{1}{(\tilde{\lambda}-\bm{\sigma})^2}\right]\lambda^2 \rho^2 \Big(1- \frac{\sigma_{\star}^2}{4\lambda^2} \Big)_{+}^2. \nonumber 
 \end{align}
Thus using \cite[Lemma 7.2]{arous2001AgingSpherical}, 
\begin{align}
    \lim_{t \uparrow \infty} \exp(-2 \tilde{\lambda} t) h(t) = \frac{2\tilde{\lambda}}{2 \tilde{\lambda} - \beta^{-1} \mathcal{S} (\tilde{\lambda}) }\mathbb{E}\left[\frac{1}{(\tilde{\lambda}-\bm{\sigma})^2}\right]\lambda^2 \rho^2 \Big(1- \frac{\sigma_{\star}^2}{4\lambda^2} \Big)_{+}^2. \nonumber 
\end{align}
Noting that $\rho>0$, we have,  

\begin{align}
    \lim_{t \uparrow \infty} \frac{g^2(t)}{h(t)} = \frac{2 \tilde{\lambda} - \beta^{-1} \mathcal{S}(\tilde{\lambda}) }{2\tilde{\lambda} \lambda}\left[\mathbb{E}\left(\frac{1}{(\tilde{\lambda}-\bm{\sigma})^2}\right)\right]^{-1}.\nonumber 
\end{align}
To simplify this, first note that $\mathcal{S}'(z) = - \mathbb{E}\left(\frac{1}{(z-\bm{\sigma})^2}\right)$. Calculating the exact value of this derivative  we obtain that
    $$- \mathbb{E}\left(\frac{1}{(z-\bm{\sigma})^2}\right) = \frac{2\sqrt{z^2-\sigma_{\star}^2}-z}{\sigma_{\star}^2 \sqrt{z^2 - \sigma_{\star}^2}}. $$
   Equating these for $z= \tilde{\lambda}$ yields that
    $$ \left[\mathbb{E}\left(\frac{1}{(\tilde{\lambda}-\sigma)^2}\right)\right]^{-1} = \lambda^2(1-\frac{\sigma_{\star}^2}{4\lambda^2}).$$
    
    Putting things together, 
    $$  \lim_{t \uparrow \infty} \frac{g^2(t)}{\lambda h(t)} = \left( 1-\frac{\beta^{-1} \mathcal{S}(\tilde{\lambda})}{2 \tilde{\lambda}}\right)\left( 1-\frac{\sigma_{\star}^2}{4\lambda^2}\right).$$
    
    Now note that $\mathcal{S}(\tilde{\lambda}) = 1/\lambda$, so the final limit becomes
    $$\lim_{t \rightarrow \infty}\frac{g^2(t)}{\lambda h(t)} = \left( 1-\frac{\beta^{-1}}{2 \lambda(\lambda + \sigma_{\star}^2/4\lambda)}\right)\left( 1-\frac{\sigma_{\star}^2}{4\lambda^2}\right). $$

It remains to analyze the sub-critical regime $2 \tilde{\lambda} < s_{\beta}$. 
Using Lemma \ref{lem:h_laplace}, we note that $\mathcal{L}_{h}(\cdot)$ has a simple pole at $s_{\beta}$. Thus there exists $C_1>0$ such that 
\begin{align}
    \lim_{z \to 0} z\mathcal{L}_{h}(s_{\beta}+z) = C_1 \neq 0. 
\end{align}
In turn, this implies 
\begin{align}
    \lim_{t \uparrow \infty} \exp(-s_{\beta}t) h(t)= C_1. \nonumber 
\end{align}

This completes the proof in this sub-case, as $\lim_{t \to \infty} g^2(t)/h(t) =0$ immediately in this case.

Finally, we focus on the case $\rho=0$. Lemma \ref{lemma:g_inversion} implies that $g=0$ in this case. On the other hand, in this case $\Phi(t) = 2\mathbb{E}[\bm{Y_0}^2] \mathbb{E}[\exp(2t \bm{\sigma})]$. Thus using Lemma \ref{lem:h_laplace}, we note that the leading asymptotics of $h(\cdot)$ is determined by the pole at $s_{\beta}$. Specifically, $\lim_{t \uparrow \infty} h(t) \neq 0$, which concludes the proof.  
 
\end{proof}

\section{Proof of Theorem \ref{thm:limits}} 
\label{sec:CK_eqns} 

 Setting $U = \bG V \in \Reals^N$, the dynamics under the rotated coordinate system $\bY_t := \bG 
 \bX_t$ can be expressed as
\begin{align}
	\label{eqn:main-equation}
	\dif{Y_t^i} = U^i \langle U, Y_t \rangle \dif{t} + \sigma^i Y_t^i \dif{t} - f'\big(\| Y_t\|^2/N \big)  Y_t^i \dif{t} + \beta^{-1/2} \dif{B_t^i}.
\end{align}

Setting $u^i = \sqrt{N} U^i$, 
we note that Eqn.~\eqref{eqn:main-equation} reduces to
\begin{align}
	\dif{Y_t^i} = u^i R^{N}(t) \dif{t} + \big( \sigma^i - f'(K^{N}(u)) \big) Y_t^i \dif{t} + \beta^{-1/2} \dif{B_t^i}. 
\end{align}

Utilizing this expression, one can verify that $Y_t^i$ takes the following integral form

\begin{align}
	\label{eqn:integral-form}
	Y_t^i &= \exp\left\{ \int_0^t \big[ \sigma^i - f'(K^{N}(s)) \big]  \dif{s} \right\} \left\{ Y_0^i+ u^i \int_0^t \exp\left\{ -  \int_0^s \big[ \sigma^i - f'(K^{N}(r)) \big] \dif{r} \right\}  R^{N}(s)   \dif{s}   \right\} \\
	& \quad + \beta^{-1/2} \int_{0}^t \exp\left\{ \int_{s}^{t} \big[ \sigma^i - f'(K^{N}(r)) \big] \dif{r}   \right\}  \dif{B^i_s}.
\end{align}

\noindent 
We first re-express the solution of our SDE. 

\begin{lemma}
\label{lemma:soln_alternative} 
Define $F_t(K,\lambda) = f'(K(t,t)) - \lambda$. 
Then we have, 
\begin{align}
    Y_t^i =& Y_0^i \exp\Big[- \int_0^t F_s(K^N, \sigma_i) \dif{s} \Big] + \int_0^t \exp\Big[ - \int_s^t F_{s_1}(K^N, \sigma_i) \dif{s_1} \Big] u_i R^N(s) \dif{s} + \beta^{-1/2} B_t^i - \nonumber \\ &\beta^{-1/2} \int_0^t B_s^i F_s(K^N, \sigma_i) \exp\Big[- \int_s^t F_{s_1} (K^N, \sigma_i) \dif{s_1}\Big] \dif{s}. \label{eq:soln_alternative}  
\end{align}
\end{lemma}

\begin{proof}
The proof follows by an application of the integration by parts formula $\int_0^t f_s \dif{B_s} = f_tB_t - \int_0^t B_s f_s' \dif{s}$ on the solution \eqref{eqn:integral-form}.  
\end{proof}

\begin{definition}
We define the empirical measure 
\begin{align}
    \nu = \frac{1}{N} \sum_i \delta_{Y_0^i, u^i , \sigma^i, B_{\bullet}^i}, 
\end{align}
where the fourth marginal is an empirical distribution on the path space $\mathcal{C}[0,T]$.
\end{definition}

Consider the following collections of functions, with domain space $\mathbb{R}^3\times\mathcal{C}[0,T] $ and range space one of $\mathcal{C}[0,T]^j$ for $j=1,2,3$: 
$$
      \mathcal{F}_j \subset \{f: \mathbb{R}^3\times\mathcal{C}[0,T] \rightarrow \mathcal{C}([0,T]^j)  \}, \,\, j=1,2,3, \,\, \text{with} $$
      \begin{align*}
     \mathcal{F}_1= \{f_j, j=1\hdots 5: f_1(Y_0,u,\sigma,B_{\bullet})(w) =uY_0\exp(w\sigma), f_2(\cdot)(w)= u^2 \exp(w\sigma), f_3(\cdot)(t)= uB_t,\\ f_4(\cdot)(w)=Y_0^2\exp(w\sigma),f_5(\cdot)(w)= \sigma Y_0\exp(w \sigma) \},\\  
     \\
\mathcal{F}_2 = \{f_6, \hdots, f_{9}: f_6(Y_0,u,\sigma,B_{\bullet})(t,w)=uB_t\exp(w\sigma),f_7(\cdot)(t,w) = uB_t\sigma\exp(w\sigma),\\ f_8(\cdot)(t,s)=B_tB_s,f_{9}(\cdot)(s,w)=Y_0B_s\exp(w\sigma) \},\\
\\
\mathcal{F}_3= \{f_{10},f_{11},f_{12}: f_{10}(Y,u,\sigma,B_{\bullet})(s_1,s_2,w) = B_{s_1}B_{s_2}\exp(w\sigma),\\
f_{11}(\cdot)(s_1,s_2,w)=\sigma B_{s_1}B_{s_2}\exp(w\sigma),f_{12}(\cdot)(s_1,s_2,w)=\sigma^2B_{s_1}B_{s_2}\exp(w\sigma) \}
 \end{align*}
 Finally, define      
\begin{align}\label{eq:Fdef}
     \mathcal{F} =\mathcal{F}_1 \cup \mathcal{F}_2 \cup \mathcal{F}_3.
     \end{align}
 All functions in $\mathcal{F}$ have the same domain. To simplify notation, for elements in $\mathcal{F}_1$, we use $f(\cdot)(w) $ to mean $f(Y_0,u,\sigma,B_{\bullet})$ evaluated at $w$ for a generic point $(Y_0,u,\sigma,B_{\bullet}) \in \mathbb{R}^3 \times \mathcal{C}[0,T]$.
Similarly for elements in $\mathcal{F}_2$ and $\mathcal{F}_3$. 

Define $\mathcal{C}_{N} = \{ \int f \dif{\nu} : f \in \mathcal{F}\}$
 and note that for $f \in \mathcal{F}_j$, $\int f d\nu \in \mathcal{C}[0,T]^j, j=1,2,3.$ We introduce the following convention: for a discrete set $S=\{s_1,\hdots,s_k \}$ and a function $f(\cdot)$, $f(S) = f(s_1,\hdots,s_k)$. 
 The next result establishes that for all $N \geq 1$, $(R_N, K_N)$ depend on $\nu$ through the statistics $\mathcal{C}_N$. 

\begin{lemma}
\label{lemma:fixed_point} 
There exist  functions
\begin{align}
    &\Phi^{(1)} : \mathcal{C}[0,T] \times \mathcal{C}([0,T]^2) \times \mathcal{C}[0,T]^{|\mathcal{F}_1|}\times \mathcal{C}([0,T]^2)^{|\mathcal{F}_2|}\times \mathcal{C}([0,T]^3)^{|\mathcal{F}_3|} \to \mathcal{C}[0,T] ,\nonumber \\
    &\Phi^{(2)}: \mathcal{C}[0,T] \times \mathcal{C}([0,T]^2) \times \mathcal{C}[0,T]^{|\mathcal{F}_1|}\times \mathcal{C}([0,T]^2)^{|\mathcal{F}_2|}\times \mathcal{C}([0,T]^3)^{|\mathcal{F}_3|} \to \mathcal{C}([0,T]^2) ,\nonumber
\end{align}
such that 
\begin{align}\label{eq:fixed-point}
  R^N = \Phi^{(1)}(R^N, K^N ,\mathcal{C}_N), \,\, K^N = \Phi^{(2)}(R^N, K^N, \mathcal{C}_N).  
\end{align}  
\end{lemma}

\begin{proof}
To this end, we combine \eqref{eq:soln_alternative} with the definition of $R_N$ to get a fixed point equation 
\begin{align}
    R^N(t) =& \frac{1}{N} \sum_{i=1}^{N} u^i Y_0^i \exp\Big[- \int_0^t (f'(K^N(s)) - \sigma^i) \dif{s} \Big] \nonumber \\
    &+ \frac{1}{N} \sum_{i=1}^{N} (u^i)^2 \int_0^t \exp\Big[- \int_s^t (f'(K^N(s_1)) - \sigma^i) \dif{s_1} \Big] R^N(s) \dif{s} \nonumber \\
    &+ \frac{\beta^{-1/2}}{N} \sum_{i=1}^{N} u^i B_t^i 
    -\frac{1}{N} \sum_{i=1}^{N} u^i \int_0^t B_s^i (f'(K^N(s)) - \sigma^i) \exp\Big[- \int_s^t (f'(K^N(s_1)) - \sigma^i) \dif{s_1} \Big]\dif{s}. \label{eq:R_representation} 
 \end{align}
 This implicitly specifies the function $\Phi^{(1)}$. The corresponding equation for $K^N$ is more involved. To track this representation systematically, we recall the representation of the solution $Y_t^i$ from \eqref{eq:soln_alternative}, and denote $Y_t^i := \sum_{j=1}^{5}T_j^i(t)$, where the $T_j^i(t)$ represent the respective terms in the RHS of \eqref{eq:soln_alternative}. Now, recall that $K^N(t,s) = \frac{1}{N} \sum_{i=1}^{N} Y_t^i Y_s^i$, and therefore 
 \begin{align}
     K^N(t,s) = \sum_{j=1}^{5} \frac{1}{N} \sum_{i=1}^{N} T_j^i(t) T_j^i(s) + \sum_{1\leq j_1<j_2\leq 5} \Big[ \frac{1}{N} \sum_{i=1}^{N} T_{j_1}^i(t) T_{j_2}^i(s) + \frac{1}{N} \sum_{i=1}^{N} T_{j_1}^i(s) T_{j_2}^i(t)\Big]. \label{eq:T_decomposition}  
\end{align}
  We argue that the RHS above is a continuous function of $R^N$, $K^N$ and $\nu$. To this end, we record these terms explicitly. Define 
  \begin{align}
      H_{\tau}^{\theta} (K) = \exp\Big[- \int_{\tau}^{\theta} f'(K(\xi)) \dif{\xi} \Big],\,\,\, DH_{\tau}^{\theta} = \frac{\dif{H_\tau^\theta(K) }}{\dif{\tau}} = f'(K(\tau)) \exp\Big[- \int_{\tau}^{\theta} f'(K(\xi) \dif{\xi} \Big]. \label{eq:h_defn} 
  \end{align}
  
  \noindent 
  First, we present the ``diagonal" terms. 
  
\begin{align}
    \frac{1}{N} \sum_{i=1}^{N} T_1^i(t) T_1^i(s) &= \frac{1}{N} \sum_{i=1}^{N} (Y_0^i)^2 \exp(\sigma^i (t+s)) H_0^t (K^N) H_0^s(K^N),  \nonumber 
    \end{align} 
    \begin{align} 
    \frac{1}{N}\sum_{i=1}^{N} T_2^i(t) T_2^i(s) &= \int_0^t \int_0^s \frac{1}{N} \sum_{i=1}^{N} (u^i)^2 \exp(\sigma^i (t+s -s_1-s_2) ) H_{s_1}^t(K^N) H_{s_2}^s(K^N) R^N(s_1) R^N(s_2) \dif{s_1} \dif{s_2},  \nonumber 
    \end{align} 
    \begin{align} 
    \frac{1}{N}\sum_{i=1}^{N} T_3^i(t) T_3^i(s) &=  \frac{\beta^{-1}}{N} \sum_{i=1}^{N} B_t^i B_s^i, \nonumber 
    \end{align} 
    \begin{align} 
    \frac{1}{N}\sum_{i=1}^{N} T_4^i(t) T_4^i(s) &= \frac{\beta^{-1}}{N} \sum_{i=1}^{N} \int_0^t \int_0^s B_{s_1}^i B_{s_2}^i DH_{s_1}^t(K^N) DH_{s_2}^t(K^N) \exp(\sigma^i (t+s -s_1 -s_2)) \dif{s_1}\dif{s_2}, \nonumber 
    \end{align}
    \begin{align} 
    \frac{1}{N}\sum_{i=1}^{N} T_5^i(t) T_5^i(s) &= \frac{\beta^{-1}}{N} \sum_{i=1}^{N} (\sigma^i)^2 \int_0^t \int_0^s B_{s_1}^i B_{s_2}^i \exp(\sigma^i (t+s - s_1 -s_2)) \dif{s_1} \dif{s_2}. \nonumber 
\end{align} 
\noindent 
Next, we present the ``off-diagonal" terms. 
\begin{align}
    \frac{1}{N}\sum_{i=1}^{N} [ T_1^i(t) T_2^i(s) + T_1^i (s) T_2^i(t)] =& H_0^t (K^N) \int_0^s \frac{1}{N} \sum_{i=1}^{N} u^i Y_0^i \exp(\sigma^i (t+s -s_1)) H_{s_1}^{s}(K^N) R^N(s_1) \dif{s_1} \nonumber \\
    +& H_0^s (K^N) \int_0^t \frac{1}{N} \sum_{i=1}^{N} u^i Y_0^i \exp(\sigma^i (t+s -s_1)) H_{s_1}^{t}(K^N) R^N(s_1) \dif{s_1}, \nonumber 
    \end{align} 
    \begin{align} 
    \frac{1}{N}\sum_{i=1}^{N} [ T_1^i(t) T_3^i(s) + T_1^i (s) T_3^i(t)] =& H_0^t(K^N) \frac{\beta^{-1/2}}{N} \sum_{i=1}^{N} Y_0^i B_s^i \exp(\sigma^i t)  + H_0^s(K^N) \frac{\beta^{-1/2}}{N} \sum_{i=1}^{N} Y_0^i B_t^i \exp(\sigma^i s), \nonumber
    \end{align}
    \begin{align} 
    \frac{1}{N}\sum_{i=1}^{N} [ T_1^i(t) T_4^i(s) + T_1^i (s) T_4^i(t)] =& - \beta^{-1/2} H_0^t(K^N) \int_0^s \frac{1}{N} \sum_{i=1}^{N} Y_0^i B_{s_1}^i \exp(\sigma^i (t+s-s_1)) DH_{s_1}^s(K^N) \dif{s_1} \nonumber \\
    &- \beta^{-1/2} H_0^s(K^N) \int_0^t \frac{1}{N} \sum_{i=1}^{N} Y_0^i B^i_{s_1}\exp(\sigma^i (t+s-s_1) ) DH_{s_1}^t(K^N) \dif{s_1} \nonumber 
    \end{align}
    \begin{align} 
    \frac{1}{N}\sum_{i=1}^{N} [ T_1^i(t) T_5^i(s) + T_1^i (s) T_5^i(t)] =& \beta^{-1/2} H_0^t(K^N) \int_0^s \frac{1}{N} \sum_{i=1}^{N} \sigma^i Y_0^i \exp(\sigma^i(t+s-s_1)) \dif{s_1} \nonumber \\
    & + \beta^{-1/2} H_0^s(K^N) \int_0^t \frac{1}{N} \sum_{i=1}^{N} \sigma^i Y_0^i \exp(\sigma^i(t+s-s_1)) \dif{s_1}, \nonumber \end{align}
    \begin{align} 
    \frac{1}{N}\sum_{i=1}^{N} [ T_2^i(t) T_3^i(s) + T_2^i (s) T_3^i(t)] =& \beta^{-1/2} \int_0^t \frac{1}{N} \sum_{i=1}^{N} u^i B^i_s \exp(\sigma^i(t-s_1)) H_{s_1}^t(K^N) R^N(s_1) \dif{s_1} \nonumber \\
    &+ \beta^{-1/2} \int_0^s \frac{1}{N} \sum_{i=1}^{N} u^i B^i_t \exp(\sigma^i(s-s_1)) H_{s_1}^t(K^N) R^N(s_1) \dif{s_1} \nonumber 
    \end{align}
    \begin{align} 
    \frac{1}{N}\sum_{i=1}^{N} [ T_2^i(t) T_4^i(s) + T_2^i (s) T_4^i(t)] =& - \beta^{-1/2} \int_0^t \int_0^{s} \frac{1}{N}\sum_{i=1}^{N} u^i B^{i}_{s_2}\exp(\sigma^i(t+s-s_1-s_2)) H^t_{s_1}(K^N) R^N(s_1) DH_{s_2}^s(K^N) \dif{s_1}\dif{s_2} \nonumber \\
    -& \beta^{-1/2} \int_0^t \int_0^{s} \frac{1}{N}\sum_{i=1}^{N} u^i B^{i}_{s_1}\exp(\sigma^i(t+s-s_1-s_2)) H^s_{s_2}(K^N) R^N(s_2) DH_{s_1}^t(K^N) \dif{s_1}\dif{s_2}, \nonumber 
    \end{align}
    \begin{align} 
    \frac{1}{N}\sum_{i=1}^{N} [ T_2^i(t) T_5^i(s) + T_2^i (s) T_5^i(t)] =& \beta^{-1/2} \int_0^t \int_0^s \frac{1}{N} \sum_{i=1}^{N} u^i \sigma^i  B^i_{s_2}\exp( \sigma^i (t+s -s_1-s_2)) H^t_{s_1}(K^N) R^N(s_1) \dif{s_1}\dif{s_2} \nonumber\\
    +&\beta^{-1/2} \int_0^t \int_0^s \frac{1}{N} \sum_{i=1}^{N} u^i \sigma^i  B^i_{s_1}\exp( \sigma^i (t+s -s_1-s_2)) H^s_{s_2}(K^N) R^N(s_2) \dif{s_1}\dif{s_2}, \nonumber
    \end{align}
    \begin{align} 
    \frac{1}{N}\sum_{i=1}^{N} [ T_3^i(t) T_4^i(s) + T_3^i (s) T_4^i(t)] =&
    - \frac{\beta^{-1}}{N} \sum_{i=1}^{N} \int_0^s B_t^i B_{s_1}^i \exp(\sigma^i (s-s_1)) DH_{s_1}^s(K^N) \dif{s_1} \nonumber \\
    &-\frac{\beta^{-1}}{N} \sum_{i=1}^{N} \int_0^t B_s^i B_{s_1}^i \exp(\sigma^i (t-s_1)) DH_{s_1}^t(K^N) \dif{s_1}\nonumber \end{align}
    \begin{align} 
    \frac{1}{N}\sum_{i=1}^{N} [ T_3^i(t) T_5^i(s) + T_3^i (s) T_5^i(t)] =& \beta^{-1} \int_0^{s} \frac{1}{N} \sum_{i=1}^{N} \sigma^i B_t^i B_{s_1}^i \exp(\sigma^i(s-s_1)) \dif{s_1} + \beta^{-1} \int_0^t \frac{1}{N} \sum_{i=1}^{N} \sigma^i B_s^i B^i_{s_1} \exp(\sigma^i(t-s_1)) \dif{s_1}, \nonumber 
    \end{align}
    \begin{align} 
    \frac{1}{N}\sum_{i=1}^{N} [ T_4^i(t) T_5^i(s) + T_4^i (s) T_5^i(t)] =& - \frac{\beta^{-1}}{N} \sum_{i=1}^{N} \sigma^i \int_0^t \int_0^s B_{s_1}^i B_{s_2}^i \exp(\sigma^i (t+s-s_1-s_2)) DH^t_{s_1}(K^N) \dif{s_1}\dif{s_2} \nonumber \\
    &- \frac{\beta^{-1}}{N} \sum_{i=1}^{N} \sigma^i \int_0^t \int_0^s B_{s_1}^i B_{s_2}^i \exp(\sigma^i (t+s-s_1-s_2)) DH^s_{s_2}(K^N) \dif{s_1}\dif{s_2}. \nonumber 
 \end{align}
 This specifies the function $\Phi^{(2)}$ implicitly. 
 
\end{proof}

Given Lemma \ref{lemma:fixed_point}, we next establish that $R^N$, $K^N$ are, in fact, functions of the low-dimensional statistics $\mathcal{C}_N$. 

\begin{lemma} 
\label{lemma:empirical_function} 
    There exist functions 
    \begin{align}
       & \Psi^{(1)}: \mathcal{C}[0,T]^{|\mathcal{F}_1|}\times \mathcal{C}([0,T]^2)^{|\mathcal{F}_2|}\times \mathcal{C}([0,T]^3)^{|\mathcal{F}_3|} \to \mathcal{C}[0,T], \nonumber\\
        & \Psi^{(2)}: \mathcal{C}[0,T]^{|\mathcal{F}_1|}\times \mathcal{C}([0,T]^2)^{|\mathcal{F}_2|}\times \mathcal{C}([0,T]^3)^{|\mathcal{F}_3|} \to \mathcal{C}([0,T]^2) \nonumber 
    \end{align}
     such that 
    $$R^N = \Psi^{(1)}(\mathcal{C}_N), \,\, K^N = \Psi^{(2)}(\mathcal{C}_N).$$ 
\end{lemma}

To this end, our main strategy is to apply a Picard iteration scheme on the fixed point equations \eqref{eq:fixed-point}. We start with some initial guess $R_{0}^N$, $K_{0}^N$, and carry out the iterative updates
\begin{align}
    R_{m+1}^{N} = \Phi^{(1)}(R_{m}^N, K_{m}^N, \mathcal{C}_N), \,\,\,\, K_{m+1}^N= \Phi^{(2)}(R_{m}^N, K_{m}^N,\mathcal{C}_N). \label{eq:picard} 
\end{align}

We will show that this iteration system is contractive, and thus identify $R^N$, $K^N$ as the unique fixed points of this system. In this endeavor, we will utilize the precise form of the functions $\Phi^{(1)}$, $\Phi^{(2)}$ as described in the proof of Lemma \ref{lemma:fixed_point}. First, we need some preliminary estimates. 
\begin{lemma}
\label{lemma:prelim_bounds} 
Recall that $f'$, defined in \eqref{eq:Ld}, is non-negative and Lipschitz.
Then we have that
\begin{itemize}
    \item[(i)] for any $m,N\geq 1$, $0\leq H^{\theta}_{\tau}(K_{m}^N) \leq 1$. 
    \item[(ii)] For any $0 \leq t \leq T$, $N,m \geq 1$, $\int_0^t |DH_u^t (K_{m}^N)| \dif{u} \leq 1 $. 
    \item[(iii)] For any $\theta \leq T$, 
    \begin{align}
        \sup_{\tau \leq \theta} |H^{\theta}_{\tau}(K_{m+1}^N) - H^{\theta}_{\tau}(K_{m}^N)| \leq \| f'\|_{L} \int_0^{\theta} |K_{m+1}^N(s,s) - K_{m}^N(s,s)| \dif{s}. 
    \end{align}
    \item[(iv)] For any $N,m \geq 1$, $0\leq \tau \leq \theta \leq T$, 
    \begin{align}
        |DH^{\theta}_{\tau} (K_{m+1}^N) - DH^{\theta}_{\tau}(K_{m}^N) | &\leq \|f'\|_L \Big[ |K_{m+1}^N(\tau, \tau) - K_{m}^N(\tau, \tau)| + (DH^{\theta}_{\tau} (K_{m+1}^N) \nonumber \\
       & + DH^{\theta}_{\tau}(K_{m}^N) ) \int_0^{\theta} | K_{m+1}^N(s,s) - K_{m}^N(s,s) | \dif{s}\Big]. \nonumber 
    \end{align}
    \item[(v)] With probability 1, there exist $C_0$, $C_1$ (possibly random), depending on $T$ such that 
    \begin{align}
        \sup_{m\geq 1} \| R_{m}^N \|_{\infty} \leq C_0 \exp{(C_1 T)}. 
    \end{align}
\end{itemize}
\end{lemma}

\begin{proof}[Proof of Lemma \ref{lemma:prelim_bounds}]
The parts (i)-(iv) are directly adapted from \cite{arous2001AgingSpherical}, and are just collected here for the convenience of the reader. We prove (v). Note that 
\begin{align}
    &R_{m+1}^N(t) = H_0^t (K_{m}^N) \frac{1}{N}\sum_{i=1}^{N} u^i Y_0^i \exp{(\sigma^i t)}+ \frac{1}{N} \sum_{i=1}^{N} (u^i)^2 \int_0^t \exp{(\sigma^i (t-s))} R_{m}^N(s) H_s^t (K_{m}^N) \dif{s}  \nonumber \\
    &+ \frac{\beta^{-1/2}}{N} \sum_{i=1}^{N} u^i B^i_t - \frac{1}{N} \sum_{i=1}^{N} u^i \int_0^t B^i_s \exp{(\sigma^i (t-s)) } DH_s^t(K_{m}^N)\dif{s} + \frac{1}{N}\sum_{i=1}^{N} u^i \sigma^i \int_0^t B_s^i \exp(\sigma^i (t-s)) H_s^t (K_{m}^N) \dif{s}. \nonumber 
\end{align}
This implies 
\begin{align}
    |R_{m+1}^N(t) | \leq& |H_0^t(K_{m}^N)| \cdot | \frac{1}{N} \sum_{i=1}^{N} u^i Y_0^i \exp(\sigma^i t) |  + \int_0^t | \frac{1}{N} \sum_{i=1}^{N} (u^i)^2 \exp(\sigma^i(t-s)) |\cdot | H_s^t(K_{(m)}^N | \cdot |R_{(m)}^N(s)| \dif{s}  + \frac{\beta^{-1/2}}{N} | \sum_{i=1}^{N} u^i B_t^i | \nonumber \\
    &+ \int_0^t |\frac{1}{N}\sum_{i=1}^{N} u^i B_s^i \exp(\sigma^i(t-s))|\cdot |DH_s^t(K_{(m)}^N)| \dif{s} + \int_0^t | \frac{1}{N} \sum_{i=1}^{N} u^i \sigma^i B_s^i \exp(\sigma^i(t-s)) | \cdot |H_s^t(K_{(m)}^N)| \dif{s}.\nonumber 
\end{align}
Thus there exists $C_0, C_1>0$ (possibly random), independent of $m$, such that 
\begin{align}
  \sup_{0\leq t \leq T}  |R_{m+1}^N(t)| \leq C_0 + C_1 \int_0^T |R_{m}^N(s)| \dif{s}. \nonumber 
\end{align}
Iterating this bound in $m$, we obtain 
\begin{align}
    \sup_{0\leq t \leq T} |R_{m+1}^N(t)| \leq C_0 \exp(C_1 T). \nonumber 
\end{align}
This completes the proof. 
\end{proof}

\noindent 
Armed with Lemma \ref{lemma:prelim_bounds}, we turn to a proof of Lemma \ref{lemma:empirical_function}. 

\begin{proof}[Proof of Lemma \ref{lemma:empirical_function}]
First observe that for any $m\geq 1$,\eqref{eq:R_representation} implies
\begin{align}
    R_{m+1}^N(t) - R_{m}^N(t) =& \frac{1}{N} \sum_{i=1}^{N} u^i Y_0^i \exp(\sigma^i t) (H_0^t(K_{m}^N) - H_0^t(K_{m-1}^N))  \nonumber \\
    +& \int_0^t \frac{1}{N}\sum_{i=1}^{N} (u^i)^2 \exp(\sigma^i(t-s)) (R_{m}^N(s) H^t_s(K_{m}^N) - R_{m-1}^N(s) H^t_s(K_{m-1}^N)) \dif{s} \nonumber \\
    -& \int_0^t \frac{1}{N}\sum_{i=1}^N u^i B_s^i \exp(\sigma^i(t-s)) (DH_s^t(K_{m}^N) - DH_s^t(K_{m-1}^N)) \dif{s}\nonumber \\
    +& \int_0^t \frac{1}{N} \sum_{i=1}^{N} u^i \sigma^i B_s^i \exp(\sigma^i (t-s)) (H_s^t (K_{m}^N) - H_s^t(K_{m-1}^N)) \dif{s}. \nonumber 
 \end{align}
 This implies, using Lemma \ref{lemma:prelim_bounds}, there exists a constant $C>0$ such that 
 \begin{align}
   &\sup_{0\leq t \leq T} | R_{m+1}^N(t) - R_{m}^N(t)| \nonumber \\
   &\leq C  \Big[ \int_0^T |K_{m}^N(\tau,\tau) - K_{m-1}^N(\tau,\tau) | \dif{\tau} + \sup_{[0,T]}\int_0^t |R_{m}^N (s) H_s^t(K_{m}^N)- R_{m-1}^N (s) H_s^t(K_{m-1}^N )| \dif{s} \Big]. \label{eq:int1}   
 \end{align}
 To control the second term, we observe, 
 \begin{align}
     &|R_{m}^N(s) H_s^t(K_{m}^N) - R_{m-1}^N(s) H_s^t(K_{m-1}^N)|\nonumber \\
    &\leq |R_{m}^N(s)| \cdot | H_s^t(K_{m}^N) - H_s^t(K_{m-1}^N)| + |H_s^t(K_{m-1}^N) | \cdot |R_{m}^N(s) - R_{m-1}^N(s)|, \nonumber   
 \end{align}
 which directly implies 
 \begin{align}
     &\sup_{[0,T]} \int_0^t |R_{m}^N(s) H_s^t(K_{m}^N) - R_{m-1}^N(s) H_s^t(K_{m-1}^N)| \dif{s} \nonumber \\
    &\lesssim \int_0^{T} |K_{m}^N(s,s) - K_{m-1}^N(s,s) | \dif{s} + \int_0^{T} |R_{m}^N(s) - R_{m-1}^N(s) | \dif{s}. \nonumber  
 \end{align}
 Plugging this back into \eqref{eq:int1}, there exists $C>0$ (independent of $m$) such that 
 \begin{align}
     \|R_{m+1}^N - R_{m}^N\|_{\infty} &:= \sup_{t \in [0,T] } |R_{m+1}^N(t) - R_{m}^N(t)| \nonumber \\
     &\leq C\Big[ \int_0^T |K_{m}^N(s,s) - K_{m-1}^N(s,s) | \dif{s} + \int_0^T |R_{m}^N(s) - R_{m-1}^N(s)| \dif{s} \Big]. \nonumber 
 \end{align}
 Analyzing the iterative update equation for $K_{m}^N$, one can derive a similar bound. 
 \begin{align}
     \| K_{m+1}^N - K_{m}^N \|_{\infty} &:= \sup_{s,t \in [0,T]^2} | K_{m}^N(s,t) - K_{m-1}^N(s,t)| \nonumber\\ 
     &\leq C\Big[ \int_0^T |K_{m}^N(s,s) - K_{m-1}^N(s,s) | \dif{s} + \int_0^T |R_{m}^N(s) - R_{m-1}^N(s)| \dif{s} \Big].
 \end{align}
 Iterating these bounds, it follows that there exists $C_1>0$ such that \begin{align}
     \max\{ \|R_{m+1}^N - R_{m}^N \|_{\infty} , \|K_{m+1}^N - K_{m}^N\|_{\infty} \} \leq 2^m C^m C_1 \frac{T^m}{m!}. \nonumber
 \end{align}
 Thus the iteration is contractive, and the fixed point system \eqref{eq:fixed-point} has a unique fixed point for each $N \geq 1$. This completes the proof. 
\end{proof}

Finally, we will establish that $\Psi^{(1)}, \Psi^{(2)}$ derived in Lemma \ref{lemma:empirical_function} are continuous functions. To this end, recall that 

\begin{align}
       & \Psi^{(1)}: \mathcal{C}[0,T]^{|\mathcal{F}_1|}\times \mathcal{C}([0,T]^2)^{|\mathcal{F}_2|}\times \mathcal{C}([0,T]^3)^{|\mathcal{F}_3|} \to \mathcal{C}[0,T], \nonumber\\
        & \Psi^{(2)}: \mathcal{C}[0,T]^{|\mathcal{F}_1|}\times \mathcal{C}([0,T]^2)^{|\mathcal{F}_2|}\times \mathcal{C}([0,T]^3)^{|\mathcal{F}_3|} \to \mathcal{C}([0,T]^2) \nonumber 
\end{align}

We equip $\mathcal{C}([0,T]^j)$, $j = 1,2,3$, with the sup-norm topology. Further, we equip $\mathcal{C}[0,T]^{|\mathcal{F}_1|}\times \mathcal{C}([0,T]^2)^{|\mathcal{F}_2|}\times \mathcal{C}([0,T]^3)^{|\mathcal{F}_3|}$ with the product topology. With this notion of convergence, we can establish the following continuity properties of $\Psi^{(1)}$ and $\Psi^{(2)}$.

\begin{lemma}
\label{lemma:continuity}
The maps $\mathscr{C} \mapsto \Psi^{(1)}(\mathscr{C})$ and $\mathscr{C} \mapsto \Psi^{(2)}(\mathscr{C})$ are continuous. 
\end{lemma}

\begin{proof}
We measure the discrepancy between $\mathscr{C}$ and $\tilde{\mathscr{C}}$ using the uniform topology on this space, and denote 
$\mathrm{d}(\mathscr{C}, \tilde{\mathscr{C}}) = \| \mathscr{C} - \tilde{\mathscr{C}} \|_{\infty}$. Define $R,K,\tilde{R}$, $\tilde{K}$ via the fixed point equations 
\begin{align}
    R = \Phi^{(1)} (R, K, \mathscr{C}), \,\,\, K= \Phi^{(2)} (R,K, \mathscr{C}), \nonumber \\
    \tilde{R} = \Phi^{(1)}(\tilde{R}, \tilde{K}, \tilde{\mathscr{C}}), \,\,\, \tilde{K} = \Phi^{(2)}(\tilde{R}, \tilde{K}, \tilde{\mathscr{C}}). \nonumber 
\end{align}
Observe that 
\begin{align}
    \|R - \tilde{R}\|_{\infty} &= \| \Phi^{(1)}(R, K, \mathscr{C}) - \Phi^{(1)}(\tilde{R}, \tilde{K}, \tilde{\mathscr{C}}) \|_{\infty} \nonumber \\
    &\leq \| \Phi^{(1)}(R,K,\mathscr{C}) - \Phi^{(1)}(\tilde{R}, \tilde{K}, \mathscr{C}) \|_{\infty}  + \| \Phi^{(1)}(\tilde{R},\tilde{K},\mathscr{C}) - \Phi^{(1)}(\tilde{R}, \tilde{K}, \tilde{\mathscr{C}}) \|_{\infty}. \nonumber 
\end{align}
Controlling the second term, we note from \eqref{eq:R_representation} that 
\begin{align}
    \|\Phi^{(1)}(\tilde{R}, \tilde{K}, \mathscr{C}) - \Phi^{(1)}(\tilde{R},\tilde{K}, \tilde{\mathscr{C}}) \|_{\infty} \leq \|\mathscr{C} - \tilde{\mathscr{C}} \|_{\infty}. \nonumber 
\end{align}
On the other hand, using the same argument as in the proof of Lemma \ref{lemma:empirical_function}, we have that 
\begin{align}
    \| \Phi^{(1)}(R,K,\mathscr{C}) - \Phi^{(1)}(\tilde{R},\tilde{K}, \mathscr{C}) \|_{\infty} \leq C \Big[ \int_0^{T} | K(s,s) - \tilde{K}(s,s) | \dif{s} + \int_0^{T} |R(s) - \tilde{R}(s)| \dif{s} \Big]. \nonumber 
\end{align}
Combining, we obtain, 
\begin{align}
    \|R - \tilde{R}\|_{\infty} \leq \| \mathscr{C} - \tilde{\mathscr{C}} \|_{\infty} + C \Big[ \int_0^{T} | K(s,s) - \tilde{K}(s,s) | \dif{s} + \int_0^{T} |R(s) - \tilde{R}(s)| \dif{s} \Big]. \nonumber 
\end{align}
Thus there exists $C'>0$ such that 
\begin{align}
    \|R - \tilde{R}\|_{\infty} \leq \| \mathscr{C} - \tilde{\mathscr{C}} \|_{\infty} \exp(C'T). \nonumber 
\end{align}
Thus $\mathscr{C}\to \Psi^{(1)}(\mathscr{C})$ is Lipschitz. 
The argument for $\Psi^{(2)}$ is analogous, and thus omitted. This completes the proof. 
\end{proof}

Our next lemma establishes that under the two initial conditions introduced in Section  \ref{sec:Introduction}, the low-dimensional statistics $\mathcal{C}_N$ converge almost surely to deterministic limits. We defer the proof of this lemma to Section \ref{sec:deterministic_proof}. 
\begin{lemma}
\label{lemma:deterministic} 
Under the i.i.d. and i.i.d. under rotated basis initial conditions, each element in $\mathcal{C}_N$ converges to deterministic limits almost surely. 
\end{lemma}

\noindent 
Finally, we complete the proof of Theorem \ref{thm:limits}, assuming Lemma \ref{lemma:empirical_function}, \ref{lemma:continuity} and Lemma \ref{lemma:deterministic}.  

\begin{proof}[Proof of Theorem \ref{thm:limits}]
Fix $T>0$. Lemma \ref{lemma:empirical_function}, \ref{lemma:continuity}  and \ref{lemma:deterministic} together imply that $R^N$, $K^N$ converge to deterministic functions $R$ and $K$ respectively. It remains to characterize these limits. Note that from \eqref{eqn:integral-form}, 
\begin{align}
    Y_t^i &= \exp\left\{ \int_0^t \big[ \sigma^i - f'(K^{N}(s)) \big]  \dif{s} \right\} \left\{ Y_0^i+ u^i \int_0^t \exp\left\{ -  \int_0^s \big[ \sigma^i - f'(K^{N}(r)) \big] \dif{r} \right\}  R^{N}(s)   \dif{s}   \right\} \\
	& \quad + \beta^{-1/2} \int_{0}^t \exp\left\{ \int_{s}^{t} \big[ \sigma^i - f'(K^{N}(r)) \big] \dif{r}   \right\}  \dif{B^i_s}.
\end{align}
We recall that $R^N(t) = \frac{1}{N} \sum_{i=1}^{N} u^i Y^i_t$ and $K^N(t,s) = \frac{1}{N} \sum_{i=1}^{N} Y_t^i Y_s^i$, and that $R^N \to R$, $K^N \to K$ uniformly. Thus calculating $R^N$, $K^N$ and setting $N \to \infty$, we obtain the desired fixed point equations in the limit. Note that the final limit operation requires that $\nu= \frac{1}{N} \sum_{i=1}^{N} \delta_{Y_0^i, u^i, \sigma^i, B_{\bullet}^i}$ converges to the claimed limits under the i.i.d. and rotated i.i.d. initial conditions. This completes the proof. 
\end{proof}

\subsection{Proof of Lemma \ref{lemma:deterministic} }
\label{sec:deterministic_proof} 

We prove Lemma \ref{lemma:deterministic} in this section. Note that for $f \in \mathcal{F}_j$, $j=1,2,3$, $\int f \dif{\nu} \in \mathcal{C}([0,T]^j)$. Thus in the lemma above, we mean specifically that $\int f \dif{\nu}$ converges almost surely as a $\mathcal{C}([0,T]^j)$-valued random variable. Throughout, we equip $\mathcal{C}([0,T]^j)$ with a sup-norm.

\begin{proof}
Before embarking on a formal proof, we summarize our general proof strategy. The proof follows in two stages--first, we establish that for any fixed $x \in [0,T]^j$, $\int f \dif{\nu} (x)$ converges almost surely. We next prove an additional Holder continuity property of $\int f \dif{\nu}$, uniformly in $n$, that allows us to bootstrap the above pointwise a.s. convergence to a functional a.s. convergence statement.

We first consider the case of the initial condition (ii). Here, almost sure convergence of $\int f \dif{\nu} (x)$ for any fixed $x \in [0,T]^j$ follows immediately from the Strong Law of Large Numbers. 
We next establish that for sufficiently large $n$, the following holds for each element of $\mathcal{C}_N$:
there exists $C,\alpha >0$, independent of $n$, such that  almost surely, for any $x,y \in [0,1]^j$, 
\begin{align}\label{eq:holder}
   | \int f \dif{\nu} (x) - \int f \dif{\nu} (y) | \leq C \| x- y\|^{\alpha} . 
\end{align}
Before we delve into the proof, we argue that  \eqref{eq:holder} suffices to establish convergence of $\int f \dif{\nu}$ to $\mathbb{E} \int f \dif{\nu}$ in sup-norm. 

To see this, first observe that for every $f \in \mathcal{F}_j, j=1,2,3$, $\mathbb{E}[\int f \dif{\nu}] \in \mathcal{C}([0,T]^j)$, by the explicit form of these functions.
Thus $\mathbb{E}[\int f \dif{\nu}](\cdot)$ is uniformly continuous on $[0,T]^j$. This implies that if we fix $\epsilon >0$, there exists $\delta_{f} > 0$ such that whenever $|t_1-t_2| \leq \delta_f$, we have that
 \[\Big|\mathbb{E}\Big[\int f \dif{\nu} \Big](t_1)-\mathbb{E}\Big[\int f \dif{\nu}\Big](t_2) \Big|\leq \epsilon.\]
Define $\delta=\min\{\delta_f: f \in \mathcal{F}\}$ and fix any $0<\delta'<\delta $. 
Let $\{x_1, \cdots, x_{\ell}\}$ denote a $\delta'$-net of $[0,1]^j$. For any $y \in [0,T]^j$, if we denote $x_t$ to be the nearest point in the net, we know that for sufficiently large $n$,

\begin{align}
\Big| \Big[\int f \dif{\nu}\Big](y) -\mathbb{E}\Big[\int f \dif{\nu} \Big](y) \Big| & \leq \Big| \Big[\int f \dif{\nu} \Big](y)- \Big[\int f \dif{\nu}\Big](x_t) \Big|\label{eq:step1}\\
& + \Big| \Big[\int f \dif{\nu} \Big](x_t)-\mathbb{E}\Big[\int f \dif{\nu}\Big](x_t) \Big|+ \Big|\mathbb{E}\Big[\int f \dif{\nu}\Big](x_t)-\mathbb{E}\Big[\int f \dif{\nu}\Big](y)\Big|\nonumber\\
& \leq C (\delta')^\alpha + 2\epsilon.
\nonumber
    \end{align}
    Since the RHS does not depend on $y$, and $\epsilon, \delta'>0$ can be arbitrary as long as $\delta' < \delta$, we have the required sup-norm convergence of $\int f \dif{\nu}$ as a function on $\mathcal{C}([0,T]^j)$. 
We shall next establish Holder continuity of $\int f \dif{\nu}$ (uniformly in $n$), that is, property \eqref{eq:holder} for each $f \in \mathcal{F}$.

To begin, let us consider $f_1$. Observe that 
\begin{align}
    \Big|\frac{1}{N} \sum_{i=1}^{N} u^i Y_{0}^i \exp(w_1 \sigma^i) - \frac{1}{N} \sum_{i=1}^{N} u^i Y_{0}^i \exp(w_2 \sigma^i) \Big| &\leq \exp(\|\sigma \|_{\infty} T)   \Big( \frac{1}{N} \sum_{i=1}^{N} |u^i| |Y_{0}^i| \Big) |w_1- w_2|\\  \nonumber 
    & \leq C |w_1-w_2|,
\end{align}
 where the last inequality is true a.s. with $C$ independent of N only for sufficiently large $n$ on using the SLLN. Thus $f_1$ is Lipschitz almost surely for sufficiently large $n$.
For the above upper bound, recall that we always have $\|\sigma_{i}\|_{\infty} \leq 2\max\{ |d_{+}|, |d_{-}|\}$ for $n$ sufficiently large. 
Similar arguments work for $f_2$,
$f_4$, $f_5$, so we skip presenting those details here. 

We next turn to $f_3$. 
 Recall that 
\begin{align}
    \frac{1}{N}\sum_{i=1}^{N} u^i B_t^i = \frac{1}{\sqrt{N}} \sum_{i=1}^N U^i B_t^i = \frac{1}{\sqrt{N}} \sum_{i=1}^{N} V^i W_t^i, \nonumber 
\end{align}
where the last equality follows from the orthogonality of the matrix $\bG$. Recall that $V^i \sim \mathcal{N}(0, \frac{\lambda}{N})$. This implies 
\begin{align}
    \mathbb{P}\Big[\sup_{0 \leq t \leq T} \Big| \sum_{i=1}^{N} V^i W_t^i  \Big| > \sqrt{N} \varepsilon  \Big] &\leq \mathbb{P}\Big[\sup_{0 \leq t \leq T} \Big| \sum_{i=1}^{N} V^i W_t^i  \Big| > \sqrt{N} \varepsilon , \sum_{i=1}^{N} (V^{i} )^2 \leq C' \Big] + \mathbb{P}\Big[\sum_{i=1}^{N} (V^{i} )^2 \geq C' \Big] \nonumber \\
    &\leq 10 \exp(-CN) 
\end{align}
using standard deviation bounds for the supremum of a Brownian motion and a Chernoff bound on $\sum_{i=1}^{N} (V^i)^2$. 
This completes the proof for $f_3$. 

Next, we move to functions in $\mathcal{F}_2$. First note that for $f_8$, we have that $\int f_8 \dif {\nu} \rightarrow \mathbb{E}\int f_8 \dif{\nu}$ on $\mathcal{C}([0,T]^2)$, by the Uniform Law of Large Numbers and properties of Brownian Motion.  
Next, we present the proof for $f_6$---the proofs for $f_7$ and $f_9$ are similar. 
To show Holder continuity of $f_6$, set $\boldsymbol{D}_{w} = \textrm{diag}(\exp(w \sigma^1) \cdots, \exp(w \sigma^N))$. Then we have, 
\begin{align}
    \frac{1}{N} \sum_{i=1}^{N} u^i B_t^{i} \exp(w \sigma^i) = \frac{1}{\sqrt{N}} V^{\top} \bG^{\top} \boldsymbol{D}_w \bG \bW_t. \nonumber 
\end{align}
Thus for any $(t_1, w_1), (t_2 ,w_2) \in [0,T]^2$, 
\begin{align}
    &\Big| \frac{1}{\sqrt{N}} V^{\top} \bG^{\top} \boldsymbol{D}_{w_1} \bG \bW_{t_1} - \frac{1}{\sqrt{N}} V^{\top} \bG^{\top} \boldsymbol{D}_{w_2} \bG \bW_{t_2} \Big|  \nonumber \\
    &\leq \frac{1}{\sqrt{N}}\Big| V^{\top} \bG^{\top} \boldsymbol{D}_{w_1} \bG (\bW_{t_1} - \bW_{t_2})  \Big| + \frac{1}{\sqrt{N}} | V^{\top} \bG^{\top} (\boldsymbol{D}_{w_1} - \boldsymbol{D}_{w_2}) \bG \bW_{t_2} | \nonumber \\
    &\leq \| \bG V \|_2 \| \boldsymbol{D}_{w_1} \|_2 \frac{1}{\sqrt{N}} \| \bW_{t_1} - \bW_{t_2} \|_2 + \| \bG V \|_2 \| \boldsymbol{D}_{w_1} - \boldsymbol{D}_{w_2} \|_2 \frac{1}{\sqrt{N}} \| \bW_{t_2} \|_2. \nonumber 
\end{align}
 There exists $C>0$  such that almost surely,  
$\|\bG V\|_2 = \|V\|_2< C$ , $\sup_{w \in [0,T]}\| \boldsymbol{D}_{w_1} \|_2 \leq \exp(\| \sigma \|_{\infty}T)$. Further, $\| \boldsymbol{D}_{w_1} - \boldsymbol{D}_{w_2} \|_2 \leq C | w_1 - w_2|$, and 
\begin{align}
    \sup_{t_1, t_2 \in [0,T]} \Big| \frac{1}{N} \sum_{i=1}^{N} (W_{t_1}^i - W_{t_2}^i)^2 - |t_1-t_2| \Big| \stackrel{as}{\to} 0. \nonumber 
\end{align}
Thus there exists constants $C,C'>0$ such that almost surely, 
\begin{align}
   &\Big| \frac{1}{\sqrt{N}} V^{\top} \bG^{\top} \boldsymbol{D}_{w_1} \bG \bW_{t_1} - \frac{1}{\sqrt{N}} V^{\top} \bG^{\top} \boldsymbol{D}_{w_2} \bG \bW_{t_2} \Big|  \nonumber \\
   &\leq C (\sqrt{|t_1 - t_2|} + | w_1 - w_2|) \leq C' (\sqrt{|t_1 - t_2|} + \sqrt{| w_1 - w_2|}) , \nonumber 
\end{align}
where the last inequality follows from the fact that $w_1,w_2 \in [0,T]$, so that we always have $\sqrt{|w_1 - w_2|} \leq \sqrt{2T}$.
The other functions in $\mathcal{F}_2$ may be controlled using analogous arguments. 

Finally, we move onto the functions in $\mathcal{F}_3$. We sketch the proof for $f_{12}$; the same proof works for $f_{10}$, $f_{11}$.
Set $\tilde{\boldsymbol{D}}_{w} = \mathrm{diag}( (\sigma^i)^2 \exp(w \sigma_i))$. Thus we have, for $(w,s_1,s_2), (w', s_1', s_2') \in [0,T]^3$, 
\begin{align}
   & \Big| \frac{1}{N} \sum_{i=1}^{N} (\sigma^i)^2 \exp(w \sigma^i) B_{s_1}^i B_{s_2}^i - \frac{1}{N} \sum_{i=1}^{N} (\sigma^i)^2 \exp( w' \sigma^i) B_{s_1'}^i B_{s_2'}^i \Big| \nonumber\\
   &= \frac{1}{N} \Big| \bB_{s_1}^{\top} \tilde{\boldsymbol{D}}_w \bB_{s_2}  -  \bB_{s_1'}^{\top} \tilde{\boldsymbol{D}}_{w'} \bB_{s_2'} \Big| \leq  C(\sqrt{|s_1 - s_1'|} + | w - w'| + \sqrt{|s_2 - s_2'|}) \\
  & \leq C'(\sqrt{|s_1 - s_1'|} + \sqrt{| w - w'|} + \sqrt{|s_2 - s_2'|}) \nonumber 
\end{align}
almost surely for some universal constants $C,C'>0$. This completes the proof for initial condition (ii). 

Next, we turn to the initial condition (i). The main difference now lies in the fact that the pointwise almost sure convergence no longer follows directly from SLLN for all choices of $\mathcal{F}$. First, we observe that the difference in initial conditions only affects $f_1$, $f_4$, $f_5$ and $f_9$---thus we can restrict to these specific functions. We will use the same strategy to establish functional almost sure convergence, starting from the pointwise a.s. convergence, thus we omit those details. 
We present the proof for $f_1$--the proofs for $f_4,f_5,f_9$ are similar. Fix $w \in [0,T]$. We have, 
\begin{align}
    \frac{1}{N} \sum_{i=1}^{N} u^i Y_0^i \exp(w \sigma^i) = \frac{1}{\sqrt{N}} V^{\top} \bG^{\top} \boldsymbol{D}_w \bG \bX_0, 
\end{align}
where we use $\boldsymbol{D}_w= \mathrm{diag}(\exp(w \sigma^{i}))$. We observe that 
\begin{align}
    \mathbb{E}\Big[\frac{1}{N} \sum_{i=1}^{N} u^i Y_0^i \exp(w \sigma^i) \Big] = \mathbb{E}\Big[\mathbb{E}[\frac{1}{\sqrt{N}} V^{\top} \bG^{\top} \boldsymbol{D}_{w} \bG \bX_0 | \bG, \bX_0] \Big] =0. \nonumber 
\end{align}
since entries of $V$ are i.i.d. with mean $0$. 
Now, there exists $M_0>0$ such that $\mathcal{E}_n =\{\sum_{i=1}^{N} (X_0^i)^2 < N\cdot M_0\} $ occurs eventually almost surely. Given $\bG$ and $\bX_0$, $V^{\top} \bG^{\top} \boldsymbol{D}_w \bG \bX_0 \sim \mathcal{N}(0, \lambda (\bX_0^{\top} \bG^{\top} \boldsymbol{D}_w \bG \bX_0)/N)$. Moreover, 
\begin{align}
    \frac{\bX_0^{\top} \bG^{\top} \boldsymbol{D}_w \bG \bX_0}{N} \leq \| \boldsymbol{D}_w \|_2 \frac{\sum_{i=1}^{N} (X_0^i)^2}{N}. \nonumber 
\end{align}
Thus on the event $\mathcal{E}_n$, for any $\varepsilon>0$, 
\begin{align}
    \mathbb{P}\Big[ \Big|\frac{1}{\sqrt{N}} V^{\top} \bG^{\top} \boldsymbol{D}_{w} \bG \bX_0 \Big| > \varepsilon | \bG, \bX_0  \Big]  \leq 2 \exp( - \frac{N\varepsilon^2}{2M'}) \nonumber
\end{align}
for some universal constant $M'>0$. The proof is complete on using the Borel Cantelli lemma. 

\end{proof}

\bibliographystyle{alpha}
\bibliography{ref.bib}

\end{document}